\documentclass[reqno]{amsart}
\usepackage{amsmath,amssymb,amscd}
\usepackage{geometry}

\usepackage[percent]{overpic}
\usepackage{color,hyperref}

\def\bn{\mathbf{n}}
\def\bp{\mathbf{p}}
\def\bq{\mathbf{q}}
\def\bt{\mathbf{t}}
\def\bu{\mathbf{u}}
\def\bv{\mathbf{v}}
\def\bx{\mathbf{x}}
\def\by{\mathbf{y}}

\def\bP{\mathbb{P}}
\def\bR{\mathbb{R}}
\def\bT{\mathbb{T}}
\def\bZ{\mathbb{Z}}

\def\cA{\mathcal{A}}
\def\cB{\mathcal{B}}

\def\cD{\mathcal{D}}
\def\cE{\mathcal{E}}

\def\cK{\mathcal{K}}

\def\cO{\mathcal{O}}

\def\cR{\mathcal{R}}
\def\cS{\mathcal{S}}

\def\cU{\mathcal{U}}
\def\cV{\mathcal{V}}
\def\cW{\mathcal{W}}
\def\cX{\mathcal{X}}

\def\pa{\partial}

\def\ds{\displaystyle}

\newtheorem{them}{Theorem}[section]
\newtheorem{thm}{Theorem}
\newtheorem{cor}{Corollary}
\newtheorem{lem}[them]{Lemma}
\newtheorem{pro}[them]{Proposition}
\newtheorem{coro}[them]{Corollary}

\theoremstyle{definition}
\newtheorem{definition}{Definition}[section]
\newtheorem{remark}{Remark}[section]
\newtheorem*{example}{Example}

\numberwithin{equation}{section}

\begin{document}

\title[Convex billiards on convex spheres]{Convex billiards on convex spheres}

\author[P. Zhang]{Pengfei Zhang}
\address{Department of Mathematics, University of Houston, Houston, TX 77004.}
\email{pzhang@math.uh.edu}



\begin{abstract}
In this paper we study the dynamical billiards on a convex 2D sphere.
We investigate some generic properties
of the convex billiards on a general convex sphere.
We prove that $C^\infty$ generically,
every periodic point is either hyperbolic or
elliptic with irrational rotation number.
Moreover, every hyperbolic periodic point admits
some transverse homoclinic intersections.
A new ingredient in our approach is
Herman's result on Diophantine invariant curves  that we use
to prove the nonlinear stability of elliptic periodic points
for a dense subset of convex billiards.
\end{abstract}

\maketitle

\section{Introduction}\label{intro}

The dynamical billiards, as a class of dynamical systems,
were introduced by Birkhoff \cite{Bir17,Bir} in his study of
Lagrangian systems with two degrees of freedom.
A Lagrangian system with two degrees of freedom
is isomorphic with the motion of  a mass particle
moving on a surface rotating uniformly
about a fixed axis and carrying  a fixed conservative  field of
force with it. If the surface is not rotating and the force vanishes,
then the particle moves along geodesics on the surface.
If the surface has boundary,
then the resulting system is a billiard system.

The classical results of dynamical billiards are closely related to geometrical optics,
which has a much longer history.
For example, the discovery of the integrability of
elliptic billiards, according to Sarnak \cite{Sar}, goes back at least to Boscovich in 1757.
Surprisingly, the billiard dynamics is also related to the spectra property of Laplace--Beltrami
operator on manifolds with a boundary.
More precisely, Weyl's law in spectral theory
gives the first order asymptotic distribution
of eigenvalues of the Laplace--Beltrami operator on a bounded domain.
Weyl's conjecture on the second order asymptotic distribution was proved by Ivrii \cite{Ivr}
for any compact manifold with boundary,
under the assumption that the measure of periodic points of  billiard
dynamics on that manifold is zero.

Current study of dynamical billiard systems mainly focuses on the Euclidean case.
Birkhoff studied the dynamical billiards
inside a convex domain on the plane.
Birkhoff also conjectured that ellipses are the only integrable billiards.
A weak version of this conjecture was proved by Bialy \cite{Bia93}.
The dynamical billiards on a bounded domain with convex scatterers were
introduced by Sinai in his study
of Boltzmann Ergodic Hypothesis \cite{Sin70} on ideal gases.
Sinai discovered the dispersing mechanism and proved
that dispersing billiards are hyperbolic and ergodic.
Since then, the mathematical and physical study of chaotic billiards
has developed at a remarkable speed (see \cite{CM}),
particularly after the various defocusing mechanisms
discovered by Bunimovich \cite{Bu,Bu92},
Wojtkowski \cite{Woj}, Markarian \cite{Mar} and Donnay \cite{Don91}.
Very recently, the dynamics of some asymmetric lemon billiards
are proved to be hyperbolic \cite{BZZ}, for which the separation condition
in the defocusing mechanism was strongly violated.
See \cite{Vet,KSS,GSG} for the study of chaotic billiards
on general surfaces.
The study of chaotic billiards also provides the key idea
for the construction of hyperbolic geodesic flows on $S^2$,
see \cite{Don88a,Don88b,BuGe}.

Dynamical billiards on curved surfaces are related to the study of
quantum magnetic confinement
of non-planar 2D electron gases (2DEG) in semiconductors \cite{FLBP},
where the effect of varying the curvature
of the surface corresponds to a change in the potential energy of the system.
The dynamical billiards can be viewed as a mathematical model for
this system, and may be used to investigate the electron transport
properties of the semiconductors.
As mentioned in \cite{GSG}, the advances in semiconductor
fabrication techniques allow to manufacture solid state
(mesoscopic) devices where electrons are confined to curved surfaces.

In this paper we consider the convex billiards
on convex spheres.
Recall that the 2D sphere $S^2$ with a smooth Riemannian metric $g$
is said to be (strictly) {\it convex},
if it has positive Gaussian curvature:
$K_g(x)>0$ for all $x\in S^2$.
Given a tangent vector $\bv\in T_xS^2$,
the geodesic passing through $x$ in the direction of $\bv$
is defined by the exponential map
$\gamma_{\bv}:\bR\to S^2$, $t\mapsto\exp_x(t \bv)$.
For any two points $p,q\in S^2$, let $d(p,q)$ be the length
of the shortest geodesics connecting $p$ and $q$.
Let $\text{Inj}(S^2,g)$ be the injective radius of $(S^2,g)$.
\begin{example}
Let $S^2$ be the unit sphere in $\bR^3$
endowed with the  round metric $g_0$. Then $K_0\equiv 1$,
and every geodesic on $S^2$ moves along a great circle.
Let $p,q\in S^2$ be two points on the sphere, and $\alpha$ be the angle
between the two position vectors $\bp,\bq$.
Then the geodesic distance $d_0(p,q)$ between $p$ and $q$
is given by $d_0(p,q)=\alpha(\bp,\bq)$,
and $\cos\alpha=\langle \bp,\bq\rangle$.
Therefore, $d_0(p,q)=\arccos \langle \bp,\bq\rangle$.
Moreover, $\text{Inj}(S^2,g_0)=\pi$.
The dynamical billiards inside convex subsets
of $(S^2,g_0)$ have been studied
recently in \cite{Bol,Bia13,CP14}.
Regarding the Ivrii conjecture,
it is proved in \cite{BKNZ} that the set of periodic points of period 3
has zero measure for {\it any} billiard on the unit sphere.
\end{example}

\vskip.1in

\begin{definition}
Let $(S^2,g)$ be a convex sphere.
A closed subset $Q\subset S^2$ is said to be (geodesically) {\it convex},
if $Q$ is simply connected, and for any two points $x,y\in Q$,
there is a unique minimizing geodesic contained in $Q$
connecting $x$ and $y$.
A convex domain $Q$ is said to be {\it strictly convex},
if the interior of each minimizing geodesic
is contained in the interior $Q^o$ of $Q$.
\end{definition}
Let $Q\subset S^2$ be a convex domain,
$s$ be the arc-length parameter of $\Gamma=\pa Q$,
and $\kappa(s)$ be the geodesic
curvature of $\Gamma$ at $\Gamma(s)$.
Note that $\kappa(s)\ge 0$ for all $s$.
If $Q$ is strictly convex, then $\kappa(s)>0$ for all $s$
(except on a closed set without interior).
By definition, there are no conjugate points inside a convex domain $Q$.
In the following we require that there are no conjugate points
on the closed domain $Q$.
A sufficient condition for nonexistence of conjugate point is that
$\text{diam}(Q)<\text{Inj}(S^2,g)$.\\

The dynamical billiard on $Q$ can be defined analogously
to the planar case. That is, a particle moves along geodesics inside $Q$,
and reflects elastically upon hitting the boundary $\partial Q$.
Suppose the previous reflection happens at $\Gamma(s)$.
Let $\theta$ be the angle measured from the (positive) tangent
direction $\dot\Gamma(s)$ to the post-reflection velocity of that particle.
Then the {\it billiard map} $F$ sends $(s,\theta)$ to the next reflection
$(s_1,\theta_1)$ with $\partial Q$.
The {\it phase space} of the billiard map $F$ on $Q$ is given
by $M=\Gamma\times (0,\pi)$.
Note that the 2-form $\omega=\sin\theta\; ds\wedge d\theta$
is a symplectic form on $M$.
Let $\mu$ be the smooth probability measure
on $M$ with density $d\mu=\frac{1}{2|\partial Q|}\sin\theta\; ds\; d\theta$.

\begin{thm}\label{twist}
Let $(S^2,g)$ be a convex sphere and $Q\subset S^2$ be a
strictly convex domain with $C^r$ smooth boundary $\Gamma=\pa Q$.
Then billiard map $F: M\to M$ is a symplectic  twist map.
In particular, $F$ preserves the measure $\mu$.
\end{thm}

It is well known that a twist map has periodic orbits of Birkhoff type $(m,n)$
for all coprime pairs $(m,n)$ \cite{Bir,Ban}.
It may (most likely will) have some non-Birkhoff periodic orbits\footnote{
Take a planar elliptic billiard for example. The periodic orbits with elliptic
caustics are Birkhoff, while the periodic orbits with hyperbolic
caustics are non-Birkhoff.}.
We study some generic properties of general periodic points of dynamical billiards
on a strictly convex domain $Q$ on $(S^2,g)$.
To this end, we identify the boundary $\Gamma=\pa Q$ with the corresponding
embedding function $f:\bT\to S^2$.
Let $r\ge 2$ ($r$ could be $\infty$),
$\Upsilon^r\!(S^2,g)$ be the set of $C^r$ smooth embeddings
$\Gamma\subset S^2$ such that the enclosed domains $Q=Q(\Gamma)$
are strictly convex. Then $\Upsilon^r\!(S^2,g)$ inherits a $C^r$ topology
from $C^r\!(\bT,S^2)$.
\begin{thm}\label{KS}
There is a residual subset $\cR^r\subset \Upsilon^r\!(S^2,g)$,
such that for each $\Gamma\in \cR^r$,
the billiard map on $\Gamma$ satisfies
\begin{enumerate}
\item each periodic point is either hyperbolic,
or elliptic with irrational rotation number;

\item any two branches of invariant manifolds of hyperbolic periodic points
either do not intersect, or they have some transverse intersections.
\end{enumerate}
\end{thm}
Theorem \ref{KS} resembles the classical Kupka--Smale properties for dynamical billiards.
The abstract Kupka--Smale property is proved by applying
Thom Transversality Theorem, which requires the {\it richness}
of local perturbations. However, dynamical billiards are known
for the {\it lack} of local perturbations, since any perturbation of $\Gamma$
results in a (semi)-global perturbation of the billiard map.
See \S \ref{transver} for more details.

Given two hyperbolic periodic points $p$ and $q$, these two points
and their stable and unstable
manifolds may be separated by some KAM-type invariant curves (which are
persistent under small perturbations).
So the existence of heteroclinic intersections
may not be generic.
The following theorem answers
positively the generic existence of homoclinic intersections.
\begin{thm}\label{main}
There is a residual subset $\cR^r\subset \Upsilon^r\!(S^2,g)$,
such that for each $\Gamma\in \cR^r$,
there exist transverse homoclinic intersections
for each hyperbolic periodic point of the billiard map $F$ induced by $\Gamma$.
\end{thm}
The proof of above theorem is based on Mather's characterization \cite{Mat1}
(developed by Franks and Le Calvez in \cite{FrLC})
of the {\it prime-end extension} of
diffeomorphisms on open surfaces.
In his proof, Mather made an assumption that
each {\it elliptic fixed point},  if exists, is Moser stable.
To apply Mather's result, we have to study the elliptic periodic
points first, although the hyperbolic periodic points are the ones we are interested in.
The nonlinear stability is proved by one of Herman's results on Diophantine invariant
curves.
This property guarantees that there is no interaction between the
hyperbolic and elliptic periodic points.

Note that there are plenty of periodic points for twist maps,
and hyperbolic periodic points exist generically.
So the transverse homoclinic intersections in above theorem
do exist generically.
\begin{cor}\label{cor}
There is an open and dense subset $\cU^r\subset \Upsilon^r\!(S^2,g)$,
such that for each $\Gamma\in \cU^r$,
the billiard map on $\Gamma$ has positive topological entropy.
\end{cor}
Angenent \cite{Ang} proved that
a twist map with zero topological entropy
must have an invariant circle for each rotation number
in its rotation interval. On the other hand, invariant curves with rational
rotation numbers are fragile and can easily break up.
Therefore, the majority of twist maps should have positive
topological entropy.
So Corollary \ref{cor} can be viewed as a special case of Angenent's result.

Entropy is an important quantity indicating how chaotic a dynamical system is.
The mechanism that a transverse homoclinic intersection generates chaos
was first realized by Poincar\'e when he came across certain
nonconvergent trigonometric series during his
study of the $n$-body problem \cite{Poin}.
This mechanism was developed later by Birkhoff for the existence of
infinitely many periodic points,
and by Smale for the formulation of hyperbolic sets (horseshoe).
Poincar\'e conjectured that
for a generic $f\in\mathrm{Diff}^r_\mu(M)$, and
for every hyperbolic periodic point $p$ of $f$,
\begin{enumerate}
\item[(P1)] $W^s(p)\cap W^u(p)\backslash \{p\}\neq\emptyset$ (weaker version);

\item[(P2)] $W^s(p)\cap W^u(p)$ is dense in $W^s(p)\cup W^u(p)$.
\end{enumerate}
This is the so called Poincar\'e's connecting problem\footnote{Poincar\'e also raised
the closing problem about the denseness of periodic points, see \cite{Pug,PuRo}.}.
In the case $r=1$,  (P1) was proved by Takens in \cite{Tak}; (P2)  was proved
in \cite{Tak} on surfaces, and by Xia \cite{Xia96} in full generality.
For  $r\ge 2$, most results about this connecting problem are on
surfaces. Pixton proved in \cite{Pix} the property (P1) for planar surfaces,
by extending Robinson's result \cite{Rob73} on fixed points. For $M=\bT^2$,
(P1) was proved by Oliveira \cite{Oli1}. For general surfaces, (P1) was
proved by Oliveira \cite{Oli2} for those with irreducible homological
actions; and by Xia in \cite{Xia} for Hamiltonian diffeomorphisms.
The proof of (P1) is still not complete for general surfaces,
and there is almost no result on higher dimensions. The
property (P2) is completely open even on surfaces. For planar convex
billiards, (P1) was proved in \cite{XZ}.

Finally we make a few comments on the positive Gaussian curvature
assumption of the Riemannian metric $g$ on $S^2$.
Suppose the curvature can be negative somewhere on the sphere.
For example, one can put a small light bulb on the table $Q$ as in \cite[Fig.~2]{Don06}.
Then the {\it neck} of the light bulb will be a hyperbolic closed geodesic,
and some geodesic on its unstable manifold will hit the boundary $\Gamma$ of $Q$.
Reversing the time, we get a  billiard trajectory starting on $\Gamma$
that will not collide with $\Gamma$ in  the future.
In other words, the billiard map $F$ is not defined on the whole phase space
and is certainly not continuous. It seems that our method in this paper does
not work (at least not directly).

\section{Preliminaries}\label{prelim}

Let $(S^2,g)$ be a convex sphere,
and $Q\subset S^2$ be a strictly
convex domain with $C^r$ smooth boundary $\Gamma=\pa Q$.
Let $M\subset T_{\Gamma}S^2$ be the set of unit tangent vectors $x=(p,\bv)$
based at points $p\in\Gamma$ that point to the interior of $Q$.
Given a point $x\in M$, let $\gamma_x(t)=\exp_p(t\bv)$ be the geodesic
on $Q$ with initial condition $(\gamma(0),\dot\gamma(0))=(p,\bv)=x$.
Let $t_1$ be the next hitting time of $\gamma(t)$ with $\Gamma$,
$p_1=\gamma(t_1)\in\Gamma$,
and $x_1$ be the reflection of  $\dot \gamma(t_1)$ with respect to
the tangent line $T_{p_1}\Gamma\subset T_{p_1}S^2$.
Then the billiard map $F$ is defined as $M\to M$, $x\mapsto x_1$.
It is convenient to introduce a coordinate system on $M$. That is,
given $x=(p,\bv)\in M$,
let $s=s(p)$ be  the arc-length parameter of $\Gamma$,
$\theta=\theta(\bv)$ be the angle
of $\bv$ measured from the tangent direction $\dot\Gamma(s)$.
In the following we will represent $M$ via
this coordinate system $\{(s,\theta):s\in\Gamma,0<\theta<\pi\}$,
and rewrite the billiard map $F$ as $x=(s,\theta)\mapsto x_1=(s_1,\theta_1)$.

\subsection{Generating function of billiard map}

The dynamical billiard  has an alternative definition using the generating function.
More precisely, let $s\mapsto\Gamma(s)$ be the arc-length parameter.
We will write $s\in \Gamma$ by
identifying $s$ with $\Gamma(s)$
if there is no confusion. For example, we set
$d_\Gamma(s_1,s_2)=d(\Gamma(s_1),\Gamma(s_2))$.
Let $S(s_1,s_2)=-d_\Gamma(s_1,s_2)$,
and $\partial_i S$ be the partial derivative of $S$ with respect to $s_i$,
$i=1,2$.
We extend the generating function to an arbitrary finite segment
$(s_m,\dots,s_n)$ with $s_k\in\Gamma$, $k=m,m+1,\dots, n$,
and define the {\it action functional}
$\ds W(s_m,\dots,s_n)=\sum_{k=m}^{n-1} S(s_k,s_{k+1})$
along the segment $(s_m,\dots,s_n)$.
Such a segment is said to be an orbit segment, if
$\ds \partial_{s_k}W=\partial_2 S(s_{k-1},s_k)
+\partial_1 S(s_k,s_{k+1})=0$ for each  $k=m,\dots, n-1$.

\begin{proof}[Proof of Theorem \ref{twist}]
Given two points $s_1$ and $s_2$, let $\gamma_1(t)$ be the geodesic
from $\gamma_1(0)=\Gamma(s_1)$ to $\gamma_1(d)=\Gamma(s_2)$,
where $d=d_\Gamma(s_1,s_2)$.
Let $\theta_1$ be the angle from $\dot \Gamma(s_1)$ to $\dot\gamma_1(0)$,
and $\theta_2$ be the angle from $\dot \Gamma(s_2)$ to $\dot\gamma_1(d)$.
At $\Gamma(s_2)$, $\gamma_1$ experiences an elastic reflection,
and the new geodesic, say $\gamma_2$,
starts from $\gamma_2(0)=\Gamma(s_2)$,
such that the angle from $\dot \Gamma(s_2)$ to $\dot\gamma_2(0)$
equals $\theta_2$.
One can check that
\begin{equation}\label{generating}
\pa_1 S(s_1,s_2)=\cos\theta_1,\quad
\pa_2 S(s_1,s_2)=-\cos\theta_2.
\end{equation}
Therefore, $F(s_1,\theta_1)=(s_2,\theta_2)$ if and only if
$\pa_1 S(s_1,s_2)=\cos\theta_1$ and
$\pa_2 S(s_1,s_2)=-\cos\theta_2$.
Rewriting \eqref{generating} in total differential form, we get
$dS=\cos\theta_1 ds_1-\cos\theta_2 ds_2$. Taking exterior differential
and using $d^2 S=0$,
we get $\sin\theta_2 ds_2\wedge d\theta_2=\sin\theta_1 ds_1\wedge d\theta_1$.
Therefore, the 2-form $\omega=\sin\theta ds\wedge d\theta$ is invariant under
$F$, so is the probability measure
$d\mu=\frac{1}{2|\Gamma|}\sin\theta ds d\theta$ on $M=\Gamma\times(0,\pi)$.

To show that $F$ is a twist map on $M=\Gamma\times (0,\pi)$,
let's consider the image of $M_s=\{s\}\times (0,\pi)$ under $F$.
Let $\gamma_\theta(t)$ be the geodesic starting from $\Gamma(s)$
in the direction of $\theta$, and $t_\theta>0$ be the first moment that
$\gamma_\theta(t)$ hits $\Gamma$.
The hitting position is exactly $s_1(\theta)= p_1\circ F(s,\theta)$.
Since $Q$ is a strictly convex domain on $S^2$, the map
$s_1:(0,\pi)\to \Gamma$ is monotonically increasing.
Therefore, $F$ is a symplectic twist map on $M$.
\end{proof}

\begin{coro}
Let $\Gamma\in\Upsilon^r(S^2,g)$, and $F$ be the billiard map induced by
$\Gamma$. Then for any coprime  positive integers $(p,q)$ with $q\ge 2$,
there exists a periodic orbit $\cO_{p,q}$ of period $q$ that goes
around the table $p$ times after one period.
\end{coro}
Such an orbit $\cO_{p,q}$
is called a Birkhoff periodic orbit of type $(p,q)$.
See \cite{Bir,Ban} for more details.
Note that there may be some periodic orbits of non-Birkhoff type.

\subsection{Criterion of nondegenerate periodic orbits}

Let $W(s_1,\dots,s_n)= \sum_{k=1}^n S(s_{k-1},s_k)$ be the action
on the space of the $n$-periodic configurations $(s_k)$ in the sense that
$s_{n+k}=s_k$ for all $k$.
Then $x=(s,\theta)\in M$ is a periodic point with period $n$
if and only if $\pa_{k}W(s_1,\dots,s_n)=0$ for each $k=1,\dots, n$,
where $x_k=F^kx=(s_k,\theta_k)$ be the iterates of $x$ under the billiard map.
Given a critical $n$-periodic configuration $(s_k)$,
we let  $D^2W(s_1,\dots,s_n)=(\pa_{ij}^2W)$
be the $n\times n$ Hessian matrix of $W$
at $(s_1,\dots,s_n)$.

Let $D_xF^n$ be the tangent map at $x$ (counted to its period),
which is a $2\times 2$ matrix with determinant 1
(since $F$ preserves the symplectic form $\omega$).
Then $x$ is said to be {\it non-degenerate},
if $1$ is not an eigenvalue of $D_xF^n$.
The later condition is equivalent to $\text{Tr}(D_xF^n)\neq 2$.
Mackay and Meiss proved in  \cite{MM83}
that the trace $\text{Tr}(D_xF^n)$ is closely related to the Hessian
$D^2W$ of
$W$ at its critical path $(s_1,\dots,s_n)$.
\begin{pro}\label{hessian}
Let $\{F^k x=(s_k,\theta_k)\}$ be a periodic orbit of period $n$,
$W_2=D^2W(s_1,\dots,s_n)$ be the Hessian matrix of $W$
at $(s_1,\dots,s_n)$. Then
$\ds \mathrm{Tr}(D_xF^n)-2=(-1)^n\cdot \det(W_2)\cdot
\left(\prod_{i=1}^n S_{12}(s_{i-1},s_i)\right)^{-1}$.
\end{pro}
Note that  $\mathrm{Tr}(D_xF^n)= 2$ if and only if
$\det(W_2)= 0$. So we have the following equivalent formulations:
\begin{enumerate}
\item a periodic orbit $x=T^n x$ of the billiard map $F$ is
nondegenerate;

\item a critical cycle $(s_1,\dots,s_n)$ of the action functional $W$ is
nondegenerate.
\end{enumerate}
Birkhoff made the following observation in \cite{Bir}.
Let $(s_1,\dots,s_n)$ be an $n$-periodic configuration at where
$W$ attains its minimum. Assume the corresponding
periodic orbit $x$ is nondegenerate. Then $D^2W(s_1,\dots,s_n)$
is positive definite, and $\mathrm{Tr}(D_xF^n)-2>0$.
So the periodic point $x$ corresponding to each minimizer
 turns out to be a hyperbolic periodic point.

\subsection{Curvature and focusing time of  a tangent vector}
Now we describe some geometrical features
of the tangent map of a billiard map $F:M\to M$ on the
configuration space $S^2$ , see \cite{Vet} for more details.
We start with the coordinate system $\{(s,\theta):s\in\Gamma, \theta\in(0,\pi)\}$ on $M$,
where $s$ is the arc-length parameter of the boundary $\Gamma=\pa Q$,
and $\theta$ is the angle of a unit tangent vector $\bv\in T_{\Gamma(s)}Q$ with the
direction $\dot \Gamma(s)$.
Let $x_0=(s_0,\theta_0)\in M$, $\gamma_0(t)$ be the geodesic generated by $x_0$,
$V=a\pa_s+b\pa_\theta\in T_{x_0} M$ be a tangent vector on the phase space $M$,
and $m(V)=\frac{b}{a}$ be the slope of $V$ with respect to the $(s,\theta)$-coordinate.
Let $c:(-\epsilon,\epsilon)\to M$ is a smooth curve
passing through $c(0)=x_0$ such that $V=\dot c(0)$.
Then for each $-\epsilon<u<\epsilon$,
the point $c(u)$ will determine a  geodesic
on $Q$, say $\gamma_{u}(\cdot)$.
Putting them together, we get a beam of geodesics
around the geodesic $\gamma_0$.
A curve $\rho:(-\epsilon,\epsilon)\to S^2$ with $\rho(0)=\Gamma(s_0)$
and $\rho(u)\in \gamma_{u}$ is called
a {\it wave-front} corresponding to $V\in T_x M$,
if $\rho(u)$ is perpendicular to each $\gamma_{u}$ at $\rho(u)$.
Let $\cB(V)$ be the geodesic curvature of $\rho$ at $\rho(0)$.
Note that $\cB(V)$ does not depend on the choices of curves $c$ with $\dot c(0)=V$.

\noindent{\bf Convention.}
A wave-front has
negative curvature if it is focusing, and has
positive curvature if it is dispersing. Let $\cB(V)=\infty$
if $p$ itself is a focusing point.

Any  (infinitesimal) wave-front of billiard trajectories on $Q$
focuses at some point forward and some point backward on $S^2$
(not necessarily in $Q$),
say $p_+$ and $p_-$.
Let $f(V)=d(\Gamma(s_0),p_{+})$ be the forward focusing distance (time)
of the wavefront related to $V\in T_{x_0} M$. Set $f(V)=0$
when $\Gamma(s_0)$ itself is a focusing point of the wavefront of $V$.

Note that $\cB(V)$ and $f(V)$ can be defined via normal Jacobi
fields. That is, let $\mathbf{J}(t)=\frac{d}{du}\Big|_{u=0}\gamma_u(t)$
be the Jacobi field generated by a beam of geodesics
$\gamma_u$ along $\gamma_0$.
Jacobi fields are characterized by Jacobi equation:
$\ddot{\mathbf{J}}+R(\mathbf{J},\dot \gamma_0)\dot\gamma_0=0$,
where $R$ is the curvature tensor..
A Jacobi field $\mathbf{J}$ is said to be {\it normal},
if $\mathbf{J}(t)$ is perpendicular to $\dot \gamma(t)$
for all $t$. In this case we can write $\mathbf{J}(t)=J(t)\mathbf{n}_{t}$
for some scalar function $J(t)$, where $\mathbf{n}_{t}$ is the unit
normal vector field along $\gamma_0(t)$.
The scalar Jacobi function $J(t)$ satisfies the scalar Jacobi equation
$\ddot J+K_g\cdot J=0$,
where $K_g$ is the Gaussian curvature of $(S^2,g)$.
Note that we have $\cB(V)=\frac{\dot J(0)}{J(0)}$, $f(V)=\min\{t\ge 0: J(t)=0\}$.
So the relation between $\cB(V)$ and $f(V)$ is given by
the solution of the Jacobi equation.
For example, if $\cB(V)=0$ then the wavefront focuses at  two
{\it focal points} along the geodesic $\gamma_x$
(one forward focal point, and one backward focal point),
and these two focal points are conjugate along $\gamma_x$.

The wave-front of a vector $V$ changes its curvature
at the moment when the billiard orbit  collides with the boundary $\Gamma$.
More precisely, let $\cB^{\pm}(V)$ be the curvature of the wavefront before
and after the reflection with $\Gamma$, respectively.
The relation between the curvature $\cB^{\pm}(V)$ and the slope $m(V)$ is given by
\[m(V)=\cB^-(V)\sin\theta-\kappa(s)=\cB^+(V)\sin\theta+\kappa(s),\]
where $s= p_1(x)$ is the projection to the first coordinate of $x$.

Now let $x=(s,\theta)\in M$, $Fx=(s_1,\theta_1)$, $V\in T_x M$,
 $V_1=DF(V)\in T_{x_1}M$,
and $\rho$ be a wavefront related to $V$.
Let $\cB_t(V)$ and $f_t(V)$ be the curvature
and forward focusing time of the wavefront during
the free flight time $0<t <d_1=d_\Gamma(s,s_1)$,
$\cB^{\pm}(V_1)$ and $f^{\pm}(V_1)$ be the curvature and
focusing time right before/after the collision $t\to d_1\pm0$.
Then we have
\begin{itemize}
\item[(1).] $\cB_t(V)=\frac{\dot J(t)}{J(t)}$, where $J(t)$ is the solution of
Jacobi equation;

\item[(2).] $\ds \cB^+(V_1)=\cB^-(V_1)-\frac{2\kappa(s_1)}{\sin\theta_1}$,
where $\kappa(s_1)>0$ is the curvature at $\Gamma(s_1)$.
\end{itemize}
Item (2) is the so called Mirror Formula for geometrical optics on surfaces.
Note that $f_t(V)=f(V)-t$  when $t\le f(V)$. If $f(V)< d_\Gamma(s,s_1)$,
then the wavefront focuses
between two consecutive reflections, $\cB_t(V)$ jumps from
$-\infty$ to $+\infty$, and $f_t(V)$ jumps
from 0 to the next focusing time.

\begin{example}
In the case that $g=g_0$ is the round metric on $S^2$,
the quantities $\cB(V)$, $f(V)=d(p,p_{+})$ and
$\hat f(V)=d(p,p_{-})$
are related by the following formula:
\begin{equation}\label{g0}
f(V)+\hat f(V)=\pi,\quad  \cB(V)=-\cot f(V)=\cot \hat f (V).
\end{equation}
Let $\cB(V)=\cot \alpha_0$. Then
$\cB_t(V)=\cot (\alpha_0+t)$ for all $0\le t < d(s,s_1)$.
\begin{proof}[Proof of \eqref{g0}]
Let's consider the circles $L_\alpha$
of latitude on $S^2$ surrounding the north pole,
where $\alpha$ is the angle of the circle with the positive $z$-axis.
Then the radius of $L_\alpha$ is $r(\alpha)=\sin\alpha$,
and the geodesic curvature is $\kappa(\alpha)=\sqrt{1/r^2-1}=\cot \alpha$.
Then the results follow from the observation that
$d(p,p_{+})=\alpha$ and $d(p,p_{-})=\pi-\alpha$
(and the convention on the choices of signs of the curvature).
\end{proof}
\end{example}

\subsection{Some generic properties of periodic orbits}

Let $(S^2,g)$ be a convex sphere, $Q\subset S^2$
be a strictly convex domain, and $F:M\to M$ be the induced
billiard map on $Q$, where $M=\Gamma\times (0,\pi)$.
Note that the geodesics on Riemannian manifolds are time-reversal invariant
(this may not be true on general Finsler manifolds).
Similarly, the billiard dynamics on a convex table $Q\subset S^2$
is time-reversal invariant. More precisely, let
$\Theta:M\to M, (s,\theta)\mapsto (s,\pi-\theta)$ be the time-reversal
map. Then $F\circ \Theta= \Theta\circ F^{-1}$.
So if $\cO$ is a periodic orbit of $F$, so is $\Theta(\cO)$;
and these two orbits are distinct if $\pi/2\notin p_2(\cO)$, 
where $p_2:M\to (0,\pi)$ is the projection to the $\theta$ coordinate.
Note that $\cO$ and $\Theta(\cO)$ have the same dynamical characteristics.
We only need to consider one of them when making perturbations.

\begin{definition}
Two different periodic orbits $\cO_1$ and $\cO_2$ are said to be
{\it essentially different},
if $\cO_2$ is not the time-reversal of $\cO_1$.
\end{definition}

There are some special features for the periodic orbits on
the billiard map on $Q$ (see \cite{Sto87}):
\begin{enumerate}
\item it is possible that $|\cO(p)|\neq | p_1(\cO(p))|$:
the orbit passes some reflection point more than
once during a minimal period.

\item  it is possible that $| p_1(\cO_1\cup \cO_2)|\neq | p_1(\cO_1)|+| p_1(\cO_2)|$:
two essentially different periodic orbits have some common reflection points.
\end{enumerate}
Take the round table on standard sphere for example:
on each point $s\in\Gamma$, there exist periodic
orbits of type $(m,n)$ for all $(m,n)$.
This happens even among the orbits with the same period:
the $(1,5)$-orbit (pentagon) and the $(2,5)$-orbit (pentagram).

Before giving the precise definition,
we need to distinguish the following two cases:
symmetric and nonsymmetric orbits.
A periodic orbit $\cO(p)$ is  said to be {\it symmetric},
if $\theta_k=\pi/2$ for some $k$.
Along such an orbit, the period $n=2m$ is an even number,
the right angle reflections happen exactly twice,
and the orbit travels back and forth
between these two reflection points. See \cite{Sto87}.
A periodic orbit is said to be nonsymmetric, if it is not symmetric.
\begin{definition}
If a periodic orbit $\cO(p)$ is nonsymmetric,
then  the defect of $p$ is defined by the difference $d(p)=|\cO(p)|-| p_1(\cO(p))|$.
If $\cO(p)$ is symmetric,
then the defect of $p$ is defined by $\ds d(p)=\frac{1}{2}|\cO(p)|+1-| p_1(\cO(p))|$.
\end{definition}
\begin{figure}[h]
\begin{overpic}[height=50mm]{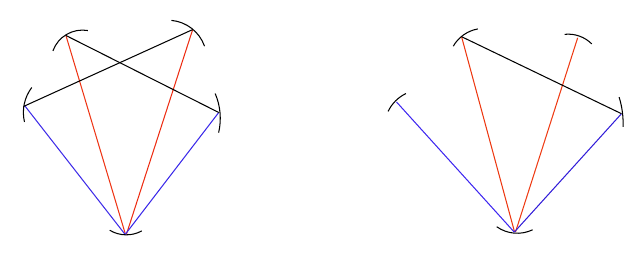}
\put(30,5){(A)}
\put(67,5){(B)}
\end{overpic}
\caption{Periodic orbits with positive defects. (A): nonsymmetric case;
(B): symmetric case. This is merely a simplistic sketch, to illustrate the two types
of defects of periodic orbits.}
\label{positivedefect}
\end{figure}
See Fig. \ref{positivedefect} for a schematic sketch of (planar) periodic orbits
with positive defect: (A) for nonsymmetric case,
and (B) for symmetric case.

\begin{pro}\label{zero}
Let $P_n(\Gamma)$ be the set of points fixed by $F^n$.
There is a residual subset $\cS_{n}\subset\Upsilon^r(S^2,g)$,
such that the following hold for the billiard map of each $\Gamma\in\cS_{n}$,
\begin{enumerate}
\item every periodic orbit in $P_n(\Gamma)$ has zero defect;

\item two essentially different periodic orbits in  $P_n(\Gamma)$
have no common reflection point.
\end{enumerate}
\end{pro}
Note that the periodic orbits of period 2 always have zero defect.
So $\cS_{2}=\Upsilon^r(S^2,g)$.

For billiards in the Euclidean domain, Proposition \ref{zero}
have been proved by Stojanov \cite{Sto87}.
Note that the following two statements are equivalent
for a given $\Gamma$:
\begin{enumerate}
\item[-] every periodic orbit has zero defect;

\item[-] any periodic path $s_0,s_1,\dots,s_n=s_0$ with positive defect
is not a billiard orbit.
\end{enumerate}
Then Proposition \ref{zero} is proved by showing that the second statement
 holds generically. The proof for billiards on $S^2$ follows
from the same idea, and is sketched in the Appendix.

\begin{remark}
Let $\cS=\bigcap_{n\ge 1} \cS_n$,
which  contains a residual subset of $\Upsilon^r(S^2,g)$.
Then for each $\Gamma\in\cS$,
\begin{enumerate}
\item[(a)] every periodic orbit of $F$ has zero defect;

\item[(b)] two different periodic orbits of $F$
do not pass any common reflection point.
\end{enumerate}
One would expect that $\cS_n$ could be open and dense, not just residual.
However, this may not be true for general domains.
In next section we will prove that the properties (a) and (b) do hold
on an open and dense subset of the convex domains in $\Upsilon^r(S^2,g)$.
\end{remark}
\begin{remark}
The following properties are obtained in \cite{PeSt87,PeSt1,PeSt2}
for billiard systems on a generic connected domain in $\bR^d$:
\begin{enumerate}
\item[(I)] the set of points fixed by $F^n$ is finite;

\item[(I\!I)] the eigenvalue of each periodic point fixed by $F^n$ is not in $\cA$,
\end{enumerate}
where $\cA$ is any countable subset of $\bR$ given in advance.
The 2D version has been obtained by Lazutkin \cite{Laz}.
We will prove that these properties hold
on an open and dense subset of convex billiards,
and the sets of points fixed by $F^n$ actually vary continuously.
This continuity plays a key role in the study of homoclinic
and heteroclinic intersections.
\end{remark}

\subsection{Parametric Transversality Theorem}
Let $M$ and $N$ be two manifolds, $K\subset M$ be a subset
and $V\subset N$ be a submanifold. A smooth map $f:M\to N$
is said to be {\it transverse} to $V$ at $x\in K$ if one of the following holds:
\begin{itemize}
\item $fx\notin V$;

\item $y=fx\in V$ and $D_xf(T_{x}M)+T_yV=T_yN$.
\end{itemize}
Then $f$ is  said to be {\it transverse} to $V$ along $ K$,
denoted by $f\pitchfork_K V$, if $f$ is transverse to $V$ at each $x\in K$.
Note that for a diffeomorphism $f\in\text{Diff}^r(M)$,
a periodic point $x$ of period $k$ is nondegenerate if and only if
the map $(\text{Id},f^n):M\to M\times M$
is transverse to the diagonal $\Delta\subset M \times M$.

Let $M$ be a smooth manifold, $D\subset \bR^q$ be an open
subset, and $\rho: D\to C^r(M,M\times M)$ be a continuous map for some $r\ge1$.
The evaluation of $\rho$, denoted by $\rho^{\text{ev}}: D\times M\to M\times M$
is given by $(v,x)\mapsto \rho(v)(x)$.
Now we can state the Parametric Transversality Theorem (see \cite{Rob95}).
\begin{them}\label{PTT}
Suppose $\rho: D\to C^r(M,M\times M)$ is continuous and
$\rho^{\text{ev}}: D\times M\to M\times M$ is $C^r$.
Let $K\subset M$ be a compact subset such that
$\rho^{\text{ev}}$ is transverse to $\Delta$ along $D\times K$.
Then the set $\{v\in D:\rho(v)\pitchfork_K \Delta\}$ is open and dense in $D$.
\end{them}

An intuitive description is also  given in \cite{Rob95}:
if there are enough parameters with which to make the necessary perturbations
at {\it one point at a time},
then the above theorem implies that the function can be approximated by one
which is transverse at {\it all points in the same time}.

\section{Perturbations of periodic points of billiard systems}\label{localperturb}

There are various types of perturbation techniques in the study of dynamical systems.
One of the widely used technique is {\it Franks' Lemma}, which allows us
to manipulate the derivatives along a periodic orbit.
The perturbations for billiard dynamics are very limited, since one
cannot perturb the billiard map $F$ directly, while the perturbation of
the underlining table changes the dynamics (semi)-globally.
See  Visscher's thesis \cite{Vis} for several results on Franks's lemma in
geometric contexts (geodesics flows and billiards).
In \cite{DOP2} the effect of the perturbation of a planar billiard system is computed
explicitly via a step by step induction.
It is difficult to generalize their approach
to dynamical billiards on surfaces with non-constant curvature.
In this section we present another proof, which uses the geometric features of
the tangent vectors of the phase space $M$
on the configuration space $S^2$.

We first give some basic definitions. Let $p$ be a periodic point of $F$ of period $n$,
$D_pF^n:T_pM\to T_pM$ be the tangent map, which can be viewed
as a matrix in $\text{SL}(2,\bR)$. Let $\lambda_p$ be an eigenvalue of $D_pF^n$.
Then $p$ is said to be {\it hyperbolic} if $|\lambda_p|\neq 1$,
be {\it parabolic} if $\lambda_p=\pm 1$, and be {\it elliptic} if otherwise.
Recall that a periodic point $p$ is said to be {\it degenerate} if $\lambda_p=1$,
and be {\it nondegenerate} if it is not degenerate.

Let $\tau(p)$ be the trace of $D_pF^n$.
Then we have the following equivalent definition:
$p$ is said to be  hyperbolic if $|\tau(p)|>2$,
be parabolic if $|\tau(p)|=2$, be elliptic if $|\tau(p)|<2$,
be degenerate if $\tau(p)=2$, and be nondegenerate if $\tau(p)\neq 2$.
All nondegenerate periodic points persist under small perturbations.

\subsection{Useful perturbations of billiard systems}
The following perturbations have been widely used in the study of generic
properties of billiards.
\begin{definition}
Let $s_0\in\Gamma$, and $I\subset\Gamma$ be a neighborhood of $s_0$.
Then a {\it normal perturbation} $\Gamma_\epsilon$
of $\Gamma$ at $s_0$ supported on $I$ is a convex curve on $S^2$ that satisfies
$\Gamma_\epsilon(s)=\Gamma(s)$ for $s=s_0$ and for $s\notin I$,
$\dot\Gamma_\epsilon(s_0)=\dot\Gamma(s_0)$,
while the curvature changes to $\kappa_\epsilon(s_0)=\kappa(s_0)+\epsilon$.
\end{definition}

The normal perturbations are essentially the only types of perturbations that preserve
the orbit $\cO(p)$, in the meanwhile, change the derivatives of $DF^n$ at $p$.
However, a degenerate periodic point may be {\it robustly degenerate} under normal
perturbations.

\begin{example}
Let $\gamma$ be a geodesic starting at a point $\bp\in S^2$,
and $\bq$ be a conjugate  point of $\bp$ along $\gamma$.
Let $Q\subset S^2$ be a convex domain containing
the geodesic segment $\gamma$ from $\bp$ to  $\bq$ as a diameter.
Then there is a periodic orbit of period 2 traveling along $\gamma$ back and forth.
Let $p=(\bp,\pi/2)$ be the corresponding point on the phase space $M$.
Then the wavefront leaving $\bp$ as a focusing point will bounce back and
forth between these two reflection points $\bp$ and $\bq$,
and focus at each reflection.
If $p$ is a degenerate periodic point for $F$,
then the degeneracy of $p$ persists under normal perturbations.
\begin{proof}
Our proof actually works for any period.
This general formulation will be used later.
Let $p$ be a periodic point such that there is no multiple reflections at
its base point $s_0$, $\Gamma_\epsilon$ be a normal perturbation
of $\Gamma$ at $s_0$.
Then for each $V\in T_pM$,
the total effect of $D_p F^n_\epsilon$ on $V$
is a shift of the curvature of the returning wave-front of $D_p F^n(V)$:
$\cB^+(D_p F^n_\epsilon(V))=\cB^+(D_p F^n(V))-\frac{2\epsilon}{\sin\theta}$,
and a shift of the slope
$m(D_p F^n_\epsilon(V))=m(D_p F^n(V))-\epsilon$.
Therefore,
$D_p F^n_\epsilon=\pm\begin{bmatrix} 1 & 0 \\ -\epsilon & 1\end{bmatrix}\circ D_p F^n$.
Then the sign is positive, since $\Gamma_\epsilon$ is a small perturbation
of $\Gamma$.

In the setting of the above example,
we denote $D_p F^2=\begin{bmatrix} a & b \\ c & d\end{bmatrix}$.
Then $b=0$ since the line $\langle\pa_\theta \rangle$ is invariant, and
$a=d=1$, since $a+d=2$ (degeneracy assumption) and $ad=1$
(symplectic property).
Therefore, $D_p F^2=\begin{bmatrix} 1 & 0 \\ c & 1\end{bmatrix}$
and $D_p F^2_\epsilon=\begin{bmatrix} 1 & 0 \\ c-\epsilon & 1\end{bmatrix}$.
This implies that  $p$ is degenerate for any normal perturbation.
\end{proof}
\end{example}

This type of persistence of degeneracy of periodic orbits (with higher periods)
may happen for the convex billiards on $S^2$ and for planar billiards.
To overcome this difficulty,
we need to consider another type of perturbations, which shift the base point
$s_0$ along the normal direction at $\Gamma(s_0)$.
It is very likely that, after the shifting perturbation,
the orbit passing through $p$ is not even closed. Luckily for us,
such a shift is only needed when the reflection at $p$ is the right angle,
and there is no multiple reflections at its base point $s_0$ within one period of $p$.
In (and only in) this case, the periodic orbit
$\cO(p)$ stays the same after the shift of $\Gamma$
along the normal direction at $\Gamma(s_0)$.

\subsection{Perturbations of periodic points}
Let $A=\begin{bmatrix} a & b \\ c & d\end{bmatrix}\in\mathrm{SL}(2,\bR)$,
and $A_\epsilon=\begin{bmatrix} 1 & 0 \\ \epsilon & 1\end{bmatrix}\circ A$.
Then $\text{Tr}(A)=a+d$ and $\text{Tr}(A_\epsilon)=a+d+\epsilon b$.
Given a periodic point $p$ of period $n$, we let
$D_pF^n=\begin{bmatrix} a_p & b_p \\ c_p & d_p\end{bmatrix}$,
and denote $\tau(p):=\text{Tr}(D_pF^n)=a_p+d_p$.

Note that the dynamics near a hyperbolic periodic point is topologically conjugate
to the linearized map $D_pF^n$ (by Hartman--Grobman Theorem)
and is well understood. However, the dynamics surrounding the degenerate
and elliptic ones are quite complicated, very sensitive
to the arithmetic properties of the linearization of $F^n$ at $p$,
and depend on the nonlinear part of $F$.
\begin{pro}\label{dege}
Let $\Gamma\in\Upsilon^r(S^2,g)$, and
Let  $p$ be a periodic point of the billiard map $F$ with zero defect.
Suppose $p$ is not hyperbolic.
Then there is a $C^{r}$ small perturbation $\Gamma_\epsilon$ of $\Gamma$
such that the trace $\tau_\epsilon(p)\neq \tau(p)$.
\end{pro}
In other words, we have the following qualitative descriptions:
\begin{enumerate}
\item if $p$ is degenerate, then  after the perturbation, it is either hyperbolic or elliptic;

\item if $p$ is elliptic, then the rotation number of $p$ can be shifted continuously
under the perturbation.
\end{enumerate}

\begin{proof}
Let $p$ be a periodic point with period $n$.
Let $\Gamma_\epsilon$ be a $C^r$ small normal perturbation of $\Gamma$ at
$s_0= p_1(p)$ which increases the curvature at $s_0$ by $\epsilon^r$.
Then we have
\[\tau_\epsilon(p)=\text{Tr}(D_pF^n_\epsilon)=a_p+d_p-\epsilon^r\cdot b_p.\]
If $b_p\neq 0$, then we are done. In the following we assume $b_p=0$.

If $b_p=0$, then we have $a_p\cdot d_p=1$, which implies $|a_p+d_p|\ge 2$.
Note that $p$ is assumed to be non-hyperbolic.
So we actually have $|a_p+d_p|=2$, and
$D_pF^n=\pm\begin{bmatrix} 1 & 0 \\ c_p & 1\end{bmatrix}$.
Then the line
$\langle\pa_\theta\rangle_p$ is fixed by $D_pF^n$.
Equivalently, the corresponding wavefront $\rho_p$ focuses at $s_0= p_1(p)$, and will
focus at $s_0$ again when it returns after one period. So we only need to
show that a small perturbation can destroy the last property (for some point on
the orbit) of $p$.

\noindent{\bf Case 1.} The orbit of $p$ is not symmetric.
Then the zero defect property implies that there is no multiple
reflection along the orbit $\cO(p)$.

\noindent{\bf Case 1a.} $c_p\neq 0$.
In this case, $\langle\pa_\theta\rangle_p$
is the only line fixed by $D_pF^n$, and $\rho_p$ is the only
invariant wavefront at $p$ and along the whole orbit of $p$.
Clearly this wavefront does not focus at $s_1= p_1(Fp)$,
since there is no conjugate point on $\Gamma$.
Therefore $\langle\pa_\theta\rangle_{Fp}$
is not fixed by $D_{Fp}F^n$, since the wavefront
corresponding to $\langle\pa_\theta\rangle_{Fp}$ focuses at $s_1$
(hence is not invariant). This implies $b_{Fp}\neq 0$,
and a normal perturbation $\Gamma_\epsilon$ of $\Gamma$ is performed at $s_1$.
Then $\tau_\epsilon(Fp)=\tau(Fp)-\epsilon^r \cdot b_{Fp}$,
and the proposition follows
since $\tau(Fp)=\tau(p)$ and $\tau_\epsilon(Fp)=\tau_\epsilon(p)$.

\noindent{\bf Case 1b.} $c_p=0$. In this case $D_pF^n=\pm I_2$. We first make
a normal perturbation at $s_0$, and get $D_pF^n_{\epsilon}\neq\pm I_2$.
Then we do another perturbation given as in Case 1a.

\noindent{\bf Case 2.}
Now we assume that the orbit of $p$ is symmetric.
Without loss of generality we assume $n>2$.
Note that there always exist multiple reflections
(even through there is no defect).
More precisely, there are exactly two simple reflections among the orbit $\cO(p)$,
and all other reflections happens exactly twice (forward and backward).
Moreover, the orbit has perpendicular reflections at these two simple reflections.

\noindent{\bf Case 2a.}
There is a wavefront that focuses exactly at those two ends.
In this case we make a small shift of $\Gamma$ along the normal direction
at one end, denote the resulting table by $\hat \Gamma$, 
so that the focused wavefront is not invariant any more.
Then one automatically gets $\hat b_p\neq 0$
and $\tau_\epsilon(p)=\tau(\hat p)-\epsilon^r\cdot \hat b_p$.

\noindent{\bf Case 2b.} No wavefront focuses exactly at those two ends.
In this case  we make a normal perturbation
at one of the two ends as we did for Case 1.
This completes the proof.
\end{proof}

\section{Kupka--Smale properties for convex billiards}\label{transver}

Let $Q\subset S^2$ be a strictly convex domain
with $C^r$ smooth boundary $\Gamma$,
$M=\Gamma\times (0,\pi)$ be the phase space
of  the billiard map $F$ induced on $Q$.
For each $n\ge 2$, let $P_n(\Gamma)\subset M$ be the set of points
fixed by $F^n$.
In the following we will show that there is an open and dense subset
$\cU_n$ such that for each $\Gamma\in\cU_n$,
$P_n(\Gamma)$ is finite and depends continuously on $\Gamma$.


Denote by $A=\bT\times (0,\pi)$ the open annulus
and by $\bar A=\bT\times [0,\pi]$ the closed annulus.
Let $\mathcal{E}(\bar A)$ be the set of positive twist homeomorphisms on $\bar A$
that {\it fix every point} on the two
boundary circles $\bT\times \{0,\pi\}$.
Let $f\in \mathcal{E}(\bar A)$, and $f_A$ be the restriction of $f$ on the open annulus $A$.
The following proposition shows that $P_n(f_A)$
cannot accumulate to the boundary
$\Gamma\times\{0,\pi\}$.

\begin{pro}\label{referee}
Let $n\ge 2$ and  $f\in \cE(\bar A)$. There exist a compact set $K\subset A$
and a small neighborhood $\cW$ of $f$ in the $C^0$ topology,
such that $P_n(g_A)\subset K$ for each $g\in \cW$.
\end{pro}
\begin{remark}
In the previous version of the paper
Proposition \ref{referee} is formulated for convex billiards,
and its proof relies heavily on some geometric feature of
billiard systems.
The current formulation of Proposition \ref{referee}
and its proof were provided to the author
in the report of one of referees.
It is clear that the current formulation is better and could be useful in other situations.
Moreover,  the proof is much shorter and easier to follow than the previous version.
The contribution from the anonymous referee is kindly acknowledged by the author.
\end{remark}
\begin{proof}
Let $n\ge 2$ and  $f\in \cE(\bar A)$ be given. Let $\tilde f$ be the unique lift
of $f$ to $\bR\times [0,\pi]$ that fixes all the points of $\bR\times\{0\}$.
There exists a neighborhood $V_0\supset \bT\times\{0\}$ such that for any
$z\in V_0$, $0\le p_1(\tilde f \tilde z)-p_1(\tilde z)< \frac{1}{3n}$, where $\hat z$
is a lift of $z$ and $p_1$
is the projection from $\bR\times [0,\pi]$ to its first coordinate.
Then there exists a small neighborhood $\cV_0$ of $f$ in $\cE(\bar A)$
such that for each $g\in \cV_0$, the lift $\tilde g$ satisfies
$0\le p_1(\tilde g \tilde z)-p_1(\tilde z)< \frac{1}{2n}$ for any lift $\tilde z$
of a point $z\in V_0$.

Pick a smaller neighborhood $W_0\subset V_0$ of $\bT\times\{0\}$ whose closure
is contained in $\bigcap_{0\le k<n}f^{-k}V_0$.
Then there exists a smaller neighborhood $\cW_0\subset \cV_0$ of $f$ in $\cE(\bar A)$
such that $W_0\subset \bigcap_{0\le k<n}g^{-k}V_0$ for each $g\in \cW_0$.
Let $g\in\cW_0$.
Then we have $0\le p_1(\tilde g^n \tilde z)-p_1(\tilde z)< \frac{1}{2}$
for each lift $\tilde z$ of $z\in W_0$. Therefore, $P_n(g_A)\cap W_0=\emptyset$
 for each $g\in\cW_0$.

Similarly we construct $W_1$ and $\cW_1$ for the other boundary component
$\bT\times\{\pi\}$, and show that $P_n(g_A)\cap W_1=\emptyset$
for each $g\in\cW_1$. Then let $K=A\backslash (W_0\cup W_1)$,
and $\cW=\cW_0\cap \cW_1$. This completes the proof.
\end{proof}

\subsection{Finiteness of $P_n(\Gamma)$ for most $\Gamma$}
A periodic point $x$ is said to be {\it non-degenerate},
if $1$ is not an eigenvalue of $D_xF^{m(x)}:T_xM\to T_xM$,
where  $m(x)$ be the minimal period  of $x$.
The minimal period of a periodic point $x\in P_n(\Gamma)$
satisfies $m(x)|n$, and may be strictly less than $n$.
Then $x$ is said to be {\it non-degenerate under $F^n$},
if $1$ is not an eigenvalue of $D_xF^{n}:T_xM\to T_xM$.

Let $\cU_n\subset\Upsilon^r(S^2,g)$ be the set of
strictly convex domains $\Gamma\in \Upsilon^r(S^2,g)$ such that
every periodic point $x\in P_n(\Gamma)$ is non-degenerate under $F^n$.
\begin{lem}\label{finite}
The set $P_n(\Gamma)$ is a finite set for each $\Gamma\in \cU_n$,
and the map $\Gamma\mapsto P_n(\Gamma)$ is continuous
on $\cU_n$.
\end{lem}
The proof is omitted since it is a classical application
of  the local inversion theorem
in the study of dynamical systems. Here we use the local uniform compactness
of the set $P_n(\Gamma)$ proved in Proposition \ref{referee}.

Note that there is no bifurcation of periodic points
in $\cU_n$. So we have the following corollary.
\begin{coro}
The cardinal map $\Gamma\in\cU_n\to |P_n(\Gamma)|$ is locally constant.
\end{coro}

Now we state the first main result of this section.
\begin{pro}\label{elem}
Let $2\le r<\infty$.
Then the set $\cU_n$ is an open and dense subset of $\Upsilon^r(S^2,g)$.
\end{pro}
The proof of  Proposition \ref{elem} consists of two parts: the openness and
the denseness of $\cU_n$. The proof of openness of $\cU_n$
is quite standard and follows easily from Lemma \ref{finite}.
The proof of the denseness of $\cU_n$ is quite long and involved.
We first give a direct proof for $n=2$ to illustrate the idea
of the proof. The proof for the general case is given after that.

\vskip.1in

Now we start to prove the denseness of  $\cU_2$. 
First we introduce a useful notation.
Given an open interval $I=(a-\epsilon,a+\epsilon)$,
the subinterval $I^r=(a-\epsilon^r,a+\epsilon^r)$ will be called the core of $I$.
\begin{proof}[Proof of the denseness of $\cU_2$]
Let $\Gamma\in\Upsilon^r(S^2,g)$ be parameterized by $\bT\to S^2$.
Given $\epsilon>0$, pick a sequence of open intervals
$I_i=(s_i-\epsilon,s_i+\epsilon)$, $1\le i\le q_2$
such that the union of their cores
$\ds \bigcup_{i=1}^{q_2} I^r_i$ covers $\bT$.
We pick $\epsilon$ small enough such that
each geodesic started on $\Gamma$ with $\theta=\pi/2$
hits each arc $I_{i}$ no more than once.
Then we cover the central line $M_{\pi/2}:=\Gamma\times\{\pi/2\}\subset M$
by finitely many open subsets $B_j\subset M$, $1\le j\le m_2$,
such that for each $j\in\{1,\dots,m_2\}$ and each $k=0,1$, there exists $i=i(j,k)$
such that
$ p_1(F^kB_j)\subset I_{i(j,k)}^r$.
From now on we fix a set $B_j\subset M$ and the corresponding index $i=i(j,0)$.

By taking $\epsilon$ smaller if necessary, we may assume that
there exists a local coordinate map $\phi_i:B(0,2)\to S^2$ around $\Gamma(I_i)$
such that $\phi_i([-1,1])\supset \Gamma(I_i)$,
where $B(0,2)\subset \bR^2$ is the 2D disk of radius $2$.
Clearly $\phi_i$ does not reflect the curvature of $\Gamma$.
Given $(s,\alpha)\in\bR^2$, let
$\Gamma_i(s,\alpha)$ be a $C^r$-small and $C^\infty$-smooth perturbation
of $\Gamma$ supported on $I_i$
(viewed in the local coordinate system $\phi_i:B(0,2)\to S^2$) that
\begin{enumerate}
\item[a).] shifts $I_i^r$ $s$ distance along
the geodesic passing through $s_i$ in the direction of $\theta=\pi/4$,

\item[b).]  then rotates the shifted piece of $I_i^r$ around its center 
by an angle $\alpha$,

\item[c).]  the complement $\Gamma\backslash I_i$ stays unchanged.
\end{enumerate}
Then we use a $C^\infty$ bump function to connect the two pieces
$I_i^r$ and $\Gamma\backslash I_i$.
Note that the exact number $\theta=\pi/4$ in Step a) is not important,
as along as $\theta\neq\pi/2$.
Since $\Upsilon^r(S^2,g)$ is open, there exists an open disk
$D_{i}\subset \bR^2$ around $(s,\alpha)=(0,0)$, such that
$\Gamma_{i}(s,\alpha)\in\Upsilon^r(S^2,g)$ and it is
$(C^r,\epsilon)$-close to $\Gamma$ for each $(s,\alpha)\in D_i$.
Let $F_{i,s,\alpha}$ be the billiard map induced by $\Gamma_i(s,\alpha)$.
This gives rise to a map $\zeta_i:(s,\alpha)\in D_i\to F_{i,s,\alpha}$,
and an evaluation map $\zeta^{\text{ev}}_i:D_i\times M\to M$,
$(s,\alpha,x)\mapsto F_{i,s,\alpha}(x)$.

Note that $\{F_{i,s,0}(x): (s,0)\in D_i\}$
and $\{F_{i,0,\alpha}(x): (0,\alpha)\in D_i\}$ are two smooth curves
passing through $F_{i,0,0}(x)=Fx$. Let $\gamma_s$
be the geodesic generated by $F_{i,s,0}(x)$, and  $\eta_\alpha$
be the geodesic generated by $F_{i,0,\alpha}(x)$, respectively.
Then
$\{\gamma_s: (s,0)\in D_i\}$
and $\{\eta_\alpha: (0,\alpha)\in D_i\}$ are two beams of geodesics
surrounding the geodesic generated by $Fx$.
Let $\mathbf{J}_1=\frac{d}{ds}\Big|_{s=0}\gamma_s$ and
$\mathbf{J}_2=\frac{d}{d\alpha}\Big|_{s=\alpha}\eta_{\alpha}$
be the corresponding Jacobi fields.
It follows from the construction of the perturbations $\Gamma_i(s,\alpha)$
that $\mathbf{J}_1$ and $\mathbf{J}_2$ are two linearly independent solutions
of the Jacobi equation.
In particular, the two curves $F_{i,\cdot,0}(x)$ and $F_{i,0,\cdot}(x)$
are transverse to each other at $Fx$.
Therefore,
\begin{equation}\label{trans2i}
D_{(0,0,x)}\zeta^{\text{ev}}_i(\bR^2\times\{0_x\})=T_{Fx}M.
\end{equation}

Note that  $F_{i,s,\alpha}\equiv F$ on the set of points not based on $I_i$.
Therefore, $F_{i,s,\alpha}^2x=F\circ F_{i,s,\alpha}x$
for any $x$ based on $I_i$ and for any $(s,\alpha)\in D_i$.
Then let us consider the graph map $\rho_{2,i}$ of $\zeta_i$ and
the corresponding evaluation map $\rho_{2,i}^{\text{ev}}$, which are given by
\[\rho_{2,i}:D_i\mapsto C^{r-1}( M, M\times M),
(s,\alpha)\mapsto (\text{Id},F^2_{i,s,\alpha}),\]
\[\rho_{2,i}^{\text{ev}}:D_i\times M\to M\times M,
(s,\alpha,x)\mapsto (x,F^2_{i,s,\alpha}(x)).\]
Let $\Delta\subset M\times M$ be the diagonal.
We need to show that $\rho_{2,i}^{\text{ev}}$ is transverse to $\Delta$
along $(0,0)\times B_j$. 
The map $\rho_{2,i}^{\text{ev}}$ is certainly transverse to $\Delta$
at the places that they do not intersect.
In the following we assume that they do intersect,
and let $x\in B_j$ be a point such that
$(x,F^2x)\in \Delta\cap \rho_{2,i}^{\text{ev}}((0,0)\times B_j)$.
In particular this implies $F^2x=x$.
Note that
\begin{equation}\label{trans2iev}
D_{(0,0,x)}\rho_{2,i}^{\text{ev}}(\bR^2\times\{0_x\})=T_{x}M\times \{0_x\},
\quad
D_{(0,0,x)}\rho_{2,i}^{\text{ev}}(\{(0,0)\}\times T_xM)=T_{(x,x)}\Delta.
\end{equation}
Clearly  $T_{x}M\times \{0_x\}$ and $T_{(x,x)}\Delta$ span $T_{(x,x)}(M\times M)$.
Therefore,
the image of $\rho_{2,i}^{\text{ev}}$ is transverse
to $\Delta\subset M\times M$ along $(0,0)\times B_j$.

Now we combine all the pieces together and define a new map
\begin{equation}\label{zeta2}
\zeta:(s_i,\alpha_i)_{i=1}^{q_2}\in \prod_{1\le i\le q_2} D_i
\mapsto F_{(s_i,\alpha_i)_{i=1}^{q_2}},
\end{equation}
such that $F_{(0^{2i-2},s_i,\alpha_i,0^{2q_2-2i})}=F_{i,s_i,\alpha_i}$.
Note that the combined perturbations may
be destructively large and even destroy the convexity of $\Gamma$.
However there does exist a small open neighborhood of $\mathbf{0}=0^{2q_2}$ in
$\prod_{1\le i\le q_2} D_i$, say $D^{2q_2}$, such that $F_{(s_i,\alpha_i)_{i=1}^{q_2}}$
is well defined and $C^{r-1}$ close $F$.
Once again, let $\zeta^{\text{ev}}:D^{2q_2}\times M\to M$
be the evaluation map of $\zeta$, let $\rho_2$ be the map from $D^{2q_2}$
to $C^{r-1}(M, M\times M)$
such that
$\rho_2\big((s_i,\alpha_i)_{i=1}^{q_2}\big)
=\big(\text{Id},F^2_{(s_i,\alpha_i)_{i=1}^{q_2}}\big)$.
We claim that the evaluation map $\rho_2^{\text{ev}}$ is transverse
to $\Delta\subset M\times M$ along $\mathbf{0}\times M$.
This is clear since for each intersection
$(x,F^2x)\in\rho_2^{\text{ev}}(\mathbf{0}\times M)\cap\Delta$,
we have $F^2x=x$ and hence $x\in M_{\pi/2}$, that is, an orbit bouncing
back and forth between two points on $\Gamma$. Then $x\in B_j$ for some
$1\le j\le m_2$. Let $i=i(j,0)$ be given such that $ p_1(B_j)\subset I_i^r$.
Then we have
\begin{equation}\label{rho2}
D_{(\mathbf{0},x)}(\rho_2^{\text{ev}})(\bR^{2q_2}\times\{0_x\})
\supset D_{(0,0,x)}(\rho^{\text{ev}}_{2,i})(\bR^{2}\times\{0_x\})=T_{x}M\times \{0_x\},
\;
D_{(\mathbf{0},x)}(\rho_2^{\text{ev}})(\mathbf{0}\times T_xM)=T_{(x,x)}\Delta.
\end{equation}
Then there exists an open neighborhood
$D\subset D^{q_2}$ of $\mathbf{0}$,
such that the combined evaluation map $\rho_2^{\text{ev}}$ is transverse to $\Delta$ along
$D\times M$.
Then by Theorem \ref{PTT},
there is a dense set subset $E\subset D$ such that for each
$(s_i,\alpha_i)_{i=1}^{q_2}\in E$,
the map $\big(\text{Id},F^2_{(s_i,\alpha_i)_{i=1}^{q_2}}\big)$
is transverse to $\Delta$ along $M_{\pi/2}$. In other words,
$\Gamma_{(s_j,\alpha_j)_{i=1}^{q_2}}\in \cU_2$ for each $(s_i,\alpha_i)_{i=1}^{q_2}\in E$, and $\Gamma$ lies in the closure of
$\{\Gamma_{(s_j,\alpha_j)_{i=1}^{q_2}}:
(s_j,\alpha_j)_{i=1}^{q_2}\in E\}\subset \cU_2$.
This shows that $\cU_2$ is dense in $\Upsilon^r(S^2,g)$.
\end{proof}

One advantage for the proof of  the denseness of $\cU_2$
is that $P_2(\Gamma)\subset M_{\pi/2}$ 
for any $\Gamma\in\Upsilon^r(S^2,g)$. 
So there is no interference when making perturbations.
For periodic orbits of higher periods, there may exist some interference
within its own orbit, since there can be some intermediate
returns to the same region on $\Gamma$ (with different directions).
So we need to take care of the possible interferences
when proving the denseness of $\cU_n$
for $n\ge 3$.
We will argue  by Strong Induction.
Suppose that we have demonstrated
the $C^r$-denseness of the open subset
$\cU_k$ for each $2\le k<n$. In the following
we will prove that the set $\cU_n$ is also  $C^r$-dense.

Let $P_n^\ast(\Gamma)$ be those periodic points
in $P_n(\Gamma)$ with minimal period less than $n$,
and $\bar P_n(\Gamma)$ be those with period exactly equal $n$.
We deal with these two parts separately.
Although a periodic point in $P_k(\Gamma)$ for $\Gamma\in\cU_k$
is non-degenerate under $F^k$,
it may  be degenerate under $F^n$.
The following lemma reduces the possible interferences from 
periodic orbits of lower periods.
\begin{lem}\label{Ukn}
Let $k<n$ with $k|n$. Then there is an open and dense subset
$\cU_{k,n}\subset \cU_k$, such that for each $\Gamma\in\cU_{k,n}$,
all periodic points in $P_k(\Gamma)$
are non-degenerate under $F^n$.
\end{lem}
\begin{proof}
It follows from the definition that $\cU_{k,n}$ is open in $\cU_k$.
So we only need to show the denseness of $\cU_{k,n}$ in $\cU_k$.
Pick $\Gamma\in\cU_k\cap \cS_k$.  Then we perturb
one reflection point on each periodic orbit $\cO(x)$, say $\Gamma_{\epsilon,x}$
such that the rotation number $\rho_\epsilon(x)$
of that orbit changes (see Lemma \ref{dege}).
Note that the new rotation number depends
continuously on the size of the perturbation.
By choosing $\epsilon(x)$ properly,
we can assume the new rotation number is irrational.
Note that $|P_k(\Gamma)|$ is locally constant and $P_k(\Gamma)$
varies continuously with respect to $\Gamma\in\cU_k$.
After a finite steps of perturbations,
the new table is in $\cU_{k,n}$.
\end{proof}

\begin{proof}[Proof of the denseness of $\cU_n$ for $n\ge 3$]
Let $\cS_{n}$ be  the open and dense
subset given by Proposition \ref{zero}, 
and $\cU_{k,n}\subset \cU_k$ be the open and dense subset
given by Lemma \ref{Ukn} for each $k<n$ and $k|n$.
Let $\ds \cU= \cS_{n}\cap \bigcap_{k<n:k|n}\cU_{k,n}$, which is also
open and dense in $\Upsilon^r(S^2,g)$.
It suffices to show that $\cU_n$ is dense in $\cU$.
Now let fix  $\Gamma\in\cU$. 
We will show that $\Gamma$ can be approximated
by a sequence of $\Gamma_i\in \cU_n$. 
The whole discussion below
will be restricted to a small neighborhood $\cW$ of $\Gamma$
given by Lemma \ref{referee}.
In particular, let $K_n$ be the uniform compact subset $K$
given there.

It is important to notice that, each periodic point
$x\in P_n^\ast(\Gamma)$ is nondegenerate under $F^n$
(since we choose $\Gamma\in\cU$),
and is isolated in $P_n(\Gamma)$.
So we can pick an open neighborhood $U\supset P_n^\ast(\Gamma)$,
such that $P_n(\hat \Gamma)\cap \overline{U}=P_n^\ast(\hat \Gamma)\subset U$
for all $\hat \Gamma$ close to $\Gamma$.
Then the function $(\text{Id},\hat F^n)$ 
is already transverse to $\Delta$ along $\overline{U}$
for all nearby $\hat \Gamma$. Hence we only need to consider the
part $K_n\backslash U$.

Let $p_2:M\to (0,\pi)$ be the projection to the second coordinate. 
Then $p_2(K_n)$ is a compact subset of $(0,\pi)$.
Without loss of generality we assume $p_2(K_n)\subset [2\theta_n, \pi-2\theta_n]$
for some $\theta_n\in(0,\pi/4)$.
The perturbations used later
in this proof will be shiftings in the direction of $\theta_n$.
Recall that for $n=2$, $P_2(\Gamma)\subset M_{\pi/2}$ for any $\Gamma$,
and we chose $\pi/4$ in the proof.

Let $x\in M$, and $\cO_n(x)=\{x,Fx\cdots,F^{n-1}x\}$
be an orbit segment of $x$ of length $n$.
Let $s_n(x)$ be the minimal separation of the 
set $\{p_1(x), p_1(Fx), \dots, p_1(F^{n-1}x)\}$ on $\Gamma$.
For example, $s_n(x)=0$ if $F^ix$ and $F^jx$ 
are reflected on the same point on $\Gamma$ for some $0\le i< j\le n-1$.
Clearly $s_n(x)>0$ for each $x\in \bar P_n(\Gamma)$, 
since $\Gamma\in \cS_n$ and every periodic orbit
in $P_n(\Gamma)$ has zero defect. 
Therefore, $s_n(x)$ can be viewed as a quantitative version of the zero defect
phenomenon.

\noindent{\bf Claim 1.} $s_n(\Gamma):=\inf\{ s_n(x):x\in \bar P_n(\Gamma)\} >0$.

\noindent{\it Proof of Claim 1.}
Suppose on the contrary that there exists $x_k\in \bar P_n(\Gamma)$
with $s_n(x_k)\to 0$. Passing to a subsequence if necessary, we assume
$x_k\to x$, which implies $F^n x=x$ and $x$ is degenerate under $F^n$.
Since every periodic point in $P_n^\ast(\Gamma)$
is nondegenerate under $F^n$,
we must have  $x\in \bar P_n(\Gamma)$ with $s_n(x)=0$.
This implies that the orbit of $x$ has positive defect,
contradicts the choice of $\Gamma\in \cS_n$.
\hfill \qed

\vskip.05in

The next claim follows directly from the fact that $s_n(x)$ depends continuously on $x$.

\noindent{\bf Claim 2.} 
There exists an open neighborhood $V_n$ of $\bar P_n(\Gamma)$,
such that $s_n(x)\ge s_n(\Gamma)/2$ for any $x\in V_n$. 

\vskip.05in

Let $s_n(\Gamma)$ be given as above, 
and $\epsilon\in(0,1/5)$ be a positive number.
Pick a sequence of open intervals
$I_i=(s_i-\epsilon\cdot s_n(\Gamma), s_i+\epsilon\cdot s_n(\Gamma))$, $1\le i\le q_n$,
such that the union of their cores
$I_i^r=(s_i-\epsilon^r\cdot s_n(\Gamma), s_i+\epsilon^r\cdot s_n(\Gamma))$
covers $\Gamma$.
Then we can cover $K_n\backslash U$ by much smaller balls
$\{B_j:j=1,\dots,m_n\}$ such that $ p_1(F^kB_j)\subset I_{i}^r$
(for some $i=i(j,k)\in\{1,\cdots, q_n\}$),
for each $k\in\{0,\dots,n-1\}$ and for each $j\in\{1,\cdots,m_n\}$.
Note that here one cannot require $i(j,k)\neq i(j,0)$ for every $k\in\{1,\cdots, n-1\}$.
We will fix a set $B_j$ and the corresponding index $i=i(j,0)$ for a moment
and let $j$ and $i$ vary at the final step of the proof.

The perturbations we need here
are similar to those we used for proving the denseness of $\cU_2$,
just here we shift $I_{i}^r$ in the direction of $\theta_n$,
where $\theta_n\in(0,\pi/4)$ is given such that
$p_2(K_n)\subset [2\theta_n,\pi-2\theta_n]$.
More precisely, for each $(s,\alpha)\in\bR^2$,
let $\Gamma_i(s,\alpha)$ be the perturbation
supported on $I_i$ that shifts the core part $I^r_i$
along the $\theta_n$ direction,
and then rotates the shifted $I^r_i$ around its center by an angle  $\alpha$.
There is an open neighborhood $D_i$ of $(0,0)$ such that
$\Gamma_i(s,\alpha)\in\Upsilon^r(S^2,g)$ for any $(s,\alpha)\in D_i$.
Note that $F_{i,s,\alpha}\equiv F$
on the set of points not based on $I_i$. 
Therefore $F_{i,s,\alpha}^k(x)=F^{k-1}\circ F_{i,s,\alpha}(x)$
for any $k\ge 2$ until next reflection of orbit with the segment $I_i$.
Similarly we can define
\begin{enumerate}
\item  the evaluation map
$\zeta_i^{\text{ev}}: D_i\times M\to M$,
$(s,\alpha,x)\mapsto F_{i,s,\alpha}^n(x)$,

\item  the graph map
$\rho_{n,i}: D_i\to C^{r-1}(M, M\times M)$,
$(s,\alpha)\mapsto (\text{Id},F_{i,s,\alpha}^n)$, and

\item  the related evaluation map
$\rho_{n,i}^{\text{ev}}: D_i\times M \to  M\times M$,
$(s,\alpha,x)\mapsto (x,F_{i,s,\alpha}^n x)$.
\end{enumerate}
We need to show that $\rho_{n,i}^{\text{ev}}$   
is transverse to the diagonal $\Delta$
along $(0,0)\times B_j$. 
The proof of this part is slightly different from the case $n=2$,
since here we may have intermediate returns:
$p_1(F^kB_j)\cap I_i\neq\emptyset$ for some $1\le k\le n-1$.
The map $\rho_{n,i}^{\text{ev}}$ is certainly transverse to $\Delta$
at the places that they do not intersect.
In the following we assume that they do intersect,
and let $x\in B_j$ be a point such that
$(x,F^nx)\in \Delta\cap \rho_{n,i}^{\text{ev}}((0,0)\times B_j)$.
In particular this implies $F^nx=x$, and $x\in \bar P_n(\Gamma)$.
Let $V_{n,j}=V_n\cap B_j$, where $V_n$ is given by Claim 2.
Note that $V_{n,j}$ contains a small neighborhood
of $x$ in $M$, and $s_n(y)\ge s_n(\Gamma)/2$ for any $y\in V_{n,j}$.
Therefore, $p_1(F^kV_{n,j})\cap I_i=\emptyset$
for each $1\le k\le n-1$, where $i=i(j,0)$ is fixed at the beginning of the proof.
This implies that $F^n_{{i,s,\alpha}}(y)=F^{n-1}\circ F_{i,s,\alpha}(y)$
for any $y\in V_{n,j}$. In this way we exclude the possible interference
of intermediate returns. Then the same argument as in the proof of the
case $n=2$ shows that $\rho_{n,i}^{\text{ev}}$   
is transverse to the diagonal $\Delta$ at each of their intersection points,
see Eq.~\eqref{trans2i} and \eqref{trans2iev}.
Therefore, $\rho_{n,i}^{\text{ev}}$   
is transverse to the diagonal $\Delta$
along $(0,0)\times B_j$.

Now we combine all the pieces together and define a new map
\begin{equation}\label{zetan}
\zeta:(s_i,\alpha_i)_{i=1}^{q_n}\in \prod_{1\le i\le q_n}D_i
\to F_{(s_i,\alpha_i)_{i=1}^{q_n}}.
\end{equation}
Again there exists a small neighborhood of $\mathbf{0}=0^{2q_n}$,
say $D^{2q_n}$,  such that $F_{(s_i,\alpha_i)_{i=1}^{q_n}}$ is well defined
and $C^r$ close to $F$ for each $(s_i,\alpha_i)_{i=1}^{q_n}\in D^{2q_n}$.
In the same way we define the combined map
$\rho_{n}: D^{2q_n}\to C^{r-1}(M, M\times M)$
and its evaluation $\rho_{n}^{\text{ev}}: D^{2q_n}\times M \to  M\times M$.
Note that any point $x$ in the intersection
$(x,F^nx)\in \Delta\cap \rho_n^{\text{ev}}(\mathbf{0}\times M)$
satisfies $F^nx=x$, and hence $x\in P_n(\Gamma)\subset K_n$.
Moreover, if $x\in P^\ast_n(\Gamma)$ has minimal
period less than $n$, then it is already nondegenerate
with respect to $F^n$. So we are left with the case that
$x\in \bar P_n(\Gamma)$ has minimal
period exactly $n$. In this case $x\in B_j$ for some $j$.
Then using the same argument as in Eq.~\eqref{rho2}, we see that
$\rho_n^{\text{ev}}$ is transverse to $\Delta$ at
$(\mathbf{0},x)$.
Letting $x$ vary, we have that the map $\rho_n^{\text{ev}}$
is transverse to $\Delta$ along
$\mathbf{0}\times K_n$.
By the openness property of transverse intersection,
there is an open neighborhood $D\subset D^{2q_n}$ of $\mathbf{0}$
such that the map $\rho_n^{\text{ev}}$
is transverse to $\Delta$ along
$D\times K_n$. Then Theorem \ref{PTT} implies
that there exists a dense subset of parameters $E\subset D$
such that for each $(s_j,\alpha_j)_{i=1}^{q_n}\in E$,
the map $\big(\text{Id},F^n_{(s_j,\alpha_j)_{i=1}^{q_n}}\big)$
is transverse to the diagonal $\Delta$
along $K_n$. In other words, $\Gamma_{(s_j,\alpha_j)_{i=1}^{q_n}}\in \cU_n$
for each $(s_j,\alpha_j)_{i=1}^{q_n}\in E$,
and $\Gamma$ lies in the closure of
$\{\Gamma_{(s_j,\alpha_j)_{i=1}^{q_n}}:
(s_j,\alpha_j)_{i=1}^{q_n}\in E\}\subset \cU_n$.
This prove the denseness of $\cU_n$ in $\cU$ and hence in $\Upsilon^r(S^2,g)$.
\end{proof}

In the previous part of this section, we fix the regularity $r\ge 2$
and use the notation $\cU_n$.
Now we switch to $\cU_n^r$ to indicate the dependence of $\cU_n$ on
the regularity $r$.
Let $\cR^r=\bigcap_{n\ge 2}\cU_n^r$.
Recall that a periodic point is said to be {\it elementary}, if it is either
hyperbolic, or elliptic with irrational rotation number. Then we have
\begin{them}\label{KS1}
There exists a residual subset $\cR^r$ of $\Upsilon^r(S^2,g)$,
such that for each $\Gamma\in\cR^r$,
every periodic point of the billiard map induced on $\Gamma$ is  elementary.
\end{them}
\begin{remark}
The proof of Theorem \ref{KS1}
among the abstract space $\mathrm{Diff}^r_\mu(M)$ was given
in \cite{Rob70}.  Robinson's proof is based on some version of
transversality theorem. The proof using Parametric
Transversality Theorem was given later in his book \cite{Rob95}.
Generally speaking, the transversality result applies
if the perturbation space is rich enough.
This is the difficult part in the study of  dynamical billiards,
since the perturbations of the billiard  map $F$
can only be made via deformations of the billiard table $\Gamma$.
\end{remark}

Note that the proof of the denseness of Proposition \ref{elem}
does not apply to the case $r=\infty$ (at least not directly).
The dynamical nature guarantees that the genericity holds also in $C^\infty$
category.
\begin{them}\label{denseinfty}
There is a residual subset
$\cR^\infty\subset \Upsilon^\infty(S^2,g)$,
such that for each $\Gamma\in\cR^\infty$,
every periodic point of the billiard map induced on $\Gamma$ is  elementary.
\end{them}
\begin{proof}
Consider the set
$\ds \cU_n^\infty=\left(\bigcup_{r\ge n}\cU_n^r\right)\cap \Upsilon^\infty(S^2,g)$:
this set is open in $\Upsilon^\infty(S^2,g)$ and $C^r$ dense for each $r\ge n$.
Therefore $\cU_n^\infty$ is open and $C^\infty$ dense in $\Upsilon^\infty(S^2,g)$.
Let $\cR^\infty=\bigcap_{n\ge 2} \cU_n^\infty$.
\end{proof}

Let $\cV_n\subset \Upsilon^r(S^2,g)$
be the set of strictly convex domains $Q\subset S^2$
such that
\begin{itemize}
\item[(a).] each periodic orbit in $P_n(\Gamma)$ has zero defect;

\item[(b).] any two periodic orbits in $P_n(\Gamma)$ has no common reflection points.
\end{itemize}
The following proof is based on our understanding of the properties
of the billiard maps in the open and dense subset $\cU_n$ in $\Upsilon^r(S^2,g)$.
\begin{pro}
The set $\cV_n$ contains an open and dense subset of $\Upsilon^r(S^2,g)$.
\end{pro}
\begin{proof}
The denseness follows from Proposition \ref{zero}.
It suffices to show the openness of $\cV_n$ in $\cU_n$.
Let  $ p_1:M\to \Gamma$ is the projection to the first coordinate,
$s_n(\Gamma)$ be the minimal separation between
the points in $ p_1(P_n(\Gamma))\subset \Gamma$.
Then $s_n(\Gamma)>0$ for each $\Gamma\in \cU_n\cap \cR_0$.
Pick a small open neighborhood  $\cU\subset\cU_n$ on which
$|P_n(\cdot)|$ is constant and $P_n(\cdot)$ varies continuously.
Then there exists a smaller neighborhood $\cV\subset \cU$ of $\Gamma$,
such that $s_n(\hat\Gamma)>0$ for each $\hat\Gamma\in\cV_n$.
Therefore, $\cV_n$ is open in $\cU_n$.
This completes the proof.
\end{proof}

\subsection{Transverse heteroclinic intersections}

Given a hyperbolic periodic point $p$ of $F$ on $M$,
the stable manifold of $p$, $W^s(p)$ consists of points $x\in M$
that $d(F^nx,F^np)\to 0$ as $n\to\infty$. Similarly we define the unstable manifold
$W^u(p)$ of $p$. Note that both stable and unstable manifolds
are immersed curves passing through $p$.
Let $W^{s,u}_{\pm}(p)$ be the branches
of $W^{s,u}(p)\backslash \{p\}$.
Let $\cW_n\subset \Upsilon^r(S^2,g)$
be the set of  convex domains $\Gamma\in \Upsilon^r(S^2,g)$,
such that for each pair of hyperbolic periodic points $p,q\in P_n(\Gamma)$,
either $W^s(p)_{\pm}\cap W^u_{\pm}(q)=\emptyset$,
or $W^s_{\pm}(p)\pitchfork_x W^u_{\pm}(q)$
for some $x\in W^s_{\pm}(p)\cap W^u_{\pm}(q)$.
\begin{pro}\label{hetero}
The set $\cW_n$ contains an open and dense subset
of $\Upsilon^r(S^2,g)$.
\end{pro}

To prove this result, we need the following perturbation result of
Donnay \cite{Don}.
\begin{lem}\label{donnay}
Let $\Gamma\in\Upsilon^r(S^2,g)$. For each $i=\pm1$, let $x_i=F^{i}x_0$,
$c_{i}:(-\epsilon,\epsilon)\to M$ be a smooth curve with
$c_{i}(0)=x_{i}$ such that $Fc_{-1}$
does not focus at $s_0= p_1(x_0)$,
and is tangent to $F^{-1}c_1$ at $x_0$.
Then there is a $C^r$ small perturbation of $\Gamma$ at the base
point $s_0$ such that
$\hat Fc_{-1}$ and $\hat F^{-1}c_1$ are transverse at $x_0$.
\end{lem}
\begin{proof}
We consider the normal perturbations $\hat \Gamma$
with $\hat \Gamma(s_0)=\Gamma(s_0)$,
$\hat \Gamma'(s_0)=\Gamma'(s_0)$
and $\hat \kappa(s_0)=\kappa(s_0)+\epsilon$.
If the perturbation is localized at $s_0= p_1(x_0)$,
then one always has  $x_i=\hat F^{i}x_0$, and hence
$x_0\in \hat Fc_{-1}\cap\hat F^{-1}c_1$.

The nonfocusing assumption of $Fc_{-1}$  means that
$\cB^-(DF\dot c_{-1}(0))\neq\infty$,
and tangency assumption means that
$\cB^-(DF\dot c_{-1}(0))=\cB^-(DF^{-1}\dot c_1(0))$.
Suppose $\hat \kappa(s_0)\neq \kappa(s_0)$ after the perturbation.
First note that $\cB^-(D\hat F\dot c_{-1}(0))$
 and $\cB^+(D\hat F^{-1}\dot c_{1}(0))$ stay unchanged,
since these quantities do not depend on the reflection with $\hat \Gamma(s_0)$.
Then according to the Mirror Formula,
\[\cB^+(D\hat F\dot c_{-1}(0))
=\cB^-(DF\dot c_{-1}(0))-\frac{2\hat \kappa(s_0)}{\sin\theta_0}
=\cB^+(DF\dot c_{-1}(0))-\frac{2\epsilon}{\sin\theta_0}.\]
Therefore $m(D\hat F\dot c_{-1}(0))=m(D\hat F^{-1}\dot c_{1}(0))-\epsilon$,
and the intersection is transverse at $x_0$.
\end{proof}

\begin{proof}[Proof of Proposition \ref{hetero}]
We will show that $\cW_n$ contains
an open and dense subset of $\cV_n$.
Pick a small open set $\cV\subset\cV_n$ on which
$|P_n(\cdot)|$ is constant and $P_n(\cdot)$ is continuous.
It suffices to show that $\cW_n$ contains
an open and dense subset in every such $\cV$.

We enumerate $P_n(\Gamma)$ as $\{y_i(\Gamma):1\le i\le I\}$.
Given $1\le i,j\le I$, $\alpha,\beta\in\{+,-\}$,
let $\cW_{ij\alpha\beta}$ be those $\Gamma\in\cW$
such that either $W^s_{\alpha}(y_i)\cap W^u_{\beta}(y_j)=\emptyset$,
or $W^s_{\alpha}(y_i)\pitchfork_x W^u_{\beta}(y_j)$
for some $x\in W^s_{\alpha}(y_i)\cap W^u_{\beta}(y_j)$.
It suffices to show each $\cW_{ij\alpha\beta}$ contains an open and dense
subset in $\cV$,
since $\ds \bigcap\{\cW_{ij\alpha\beta}:1\le i,j\le I, \alpha,\beta\in\{+,-\}\}$
is contained in $\cW_n$.
In the following we will fix $ij$ and $\alpha\beta$.

Note that there is a simple dichotomy
for $\Gamma\in\cV$:
\begin{enumerate}
\item either there exist $\Gamma_k\to \Gamma$ such that
$W^s_{\alpha}(y_i(k))$ and $W^u_{\beta}(y_j(k))$ intersect at
some point, say $x_k$.

\item or there is a smaller neighborhood of $\Gamma$
among which $W^s_{\alpha}(y_i)$ and $W^u_{\beta}(y_j)$ do not intersect.
\end{enumerate}
It suffices to show the intersections in the first alternative can be perturbed
to be transverse. From now on we fix $\Gamma_k$
such that $W^s_{\alpha}(y_i(k))$ and $W^u_{\beta}(y_j(k))$
intersect non-transversely at $x_k$,
and drop the dependence on $k$ safely.

Note that the minimal separation $s_n(\Gamma)>0$,
and the orbit $F^kx$ approximate $y_i$ (or $y_j$) exponentially fast
as $k\to +\infty$ (or $k\to-\infty$, respectively).
By taking some iterates of $x$ 
if necessary, we can assume
that there exists an open interval $I\subset \Gamma$ of $s_0= p_1(x)$
such that all other iterates of $x$ stay out of $I$.
Now we consider the wavefront at $x$ generated
by the stable and unstable branches.
Note that there is no conjugate point in $Q$.
So no wavefront can focus at $x$ and $fx$ simultaneously.
Without loss of generality we assume they do not focus at $x$.
In particular, it implies the stable and unstable branches are not
tangent to the direction $\langle\pa_\theta\rangle$ and hence project down to
an open interval on $\Gamma$, say $I$.
Then we can make a very small perturbation of $\Gamma$
supported on $I$, such that
$W^s_{\alpha}(y_i)$ and $W^u_{\beta}(y_j)$
intersect transversely at $x$ (see Lemma \ref{donnay}).
Note that transverse intersection, once created, is robust under perturbations.
Therefore $\cW_{ij\alpha\beta}$ contains an open and dense
subset in $\cV$. This completes the proof.
\end{proof}

Let $\cR^r_{KS}=\bigcap_{n\ge 2}\cW_n$, which contains a residual
subset of  $\Upsilon^r(S^2,g)$.
\begin{them}\label{KS2}
There is a residual subset  $\cR^r_{KS}$ of $\Upsilon^r(S^2,g)$,
such that for each $\Gamma\in\cR^r_{KS}$,
\begin{enumerate}
\item every periodic point of $F$ is elementary;

\item for any two hyperbolic branches
$W^s_{\alpha}(p)$ and $W^u_{\beta}(q)$,
\begin{enumerate}
\item[(2a)] either $W^s_{\alpha}(p)\cap W^u_{\beta}(p)=\emptyset$,
\item[(2b)] or $W^s_{\alpha}(p)\pitchfork_x W^u_{\beta}(q)$
for some $x\in W^s_{\alpha}(p)\cap W^u_{\beta}(q)$.
\end{enumerate}
\end{enumerate}
\end{them}
The case $r=\infty$ can be obtained in the same way as we did for Theorem
\ref{denseinfty}.
\begin{remark}
The properties of the above theorem resemble the Kupka--Smale properties for convex billiards.
However, the above theorem does not claim that
$W^s_{\alpha}(p)$ and $W^u_{\beta}(q)$
are transverse (see \cite{CP02}), neither that
$W^s_{\alpha}(p)$ and $W^u_{\beta}(q)$ have nontrivial intersection.
In general, $W^s_{\alpha}(p)$ and $W^u_{\beta}(q)$
may be separated by
some (KAM) invariant curves, and this separation is persistent under perturbations.
In next section we will study the case when $p=q$ and
prove the generic existence
of homoclinic intersections.
\end{remark}

\section{Homoclinic intersections for hyperbolic periodic points}\label{homoprime}

In this section we study the existence of homoclinic intersections
of hyperbolic periodic points of convex billiards on $(S^2,g)$.
Our main result is the following.
\begin{pro}\label{homo1n}
There is an open and dense subset
$\cX_n\subset  \Upsilon^r(S^2,g)$
such that for each $\Gamma\in \cX_n$,
there exist transverse homoclinic intersections for each
hyperbolic periodic point $p\in P_n(\Gamma)$.
\end{pro}
It suffices to show such $\cX_n$ is open and dense in $\cW_n$
(see Proposition \ref{hetero} for the set $\cW_n$).
Note that $P_n(\Gamma)$ is finite and depends continuously for $\Gamma\in\cV_n$,
and the existence of transverse intersections is an open condition.
Then $\cX_n$ is automatically open in $\cW_n$.
So it suffices to show the $C^r$ denseness of $\cX_n$ in $\cW_n$.
\begin{remark}
A simple fact that we will use repeatedly in this section is that
$\Upsilon^\infty(S^2,g)$ is $C^r$ dense in $\Upsilon^r(S^2,g)$
for any $r\ge 2$.
For example, the perturbations constructed
in Sect.~\ref{transver} are {\it always $C^\infty$},
although they are {\it only $C^r$-small}.
Therefore, we only need to show that the $C^\infty$
smooth ones in $\cX_n$ are already $C^r$  dense  in $\cW_n$.
So in the following all the convex tables will be assumed to be $C^\infty$,
and the perturbations will always be $C^\infty$ smooth
although they are only $C^r$ small in topology.
\end{remark}
Before giving the proof, we need some preparations
to cut off the connections between  the elliptic periodic points
and the hyperbolic periodic points of $F$.

\subsection{Nonlinear stability of elliptic periodic points}
Let $f\in\mathrm{Diff}^\infty_\mu(M)$ and $p$ be a fixed point of $f$.
An elliptic fixed point is also said to be {\it linearly stable}.
Then a fixed point $p$ is said to be (nonlinearly) {\it stable},
if there are nesting closed disks
$\{D_n\}$ with $p\in D_{n+1}\subset D^o_n$ such that
$\bigcap_{n\ge 1} D_n =\{p\}$ and $f|_{\partial D_n}$ is transitive.
Note that stable fixed points are isolated from the dynamics,
and any invariant rays either coincides with some of those $\pa D_n$,
or are disjoint from $\pa D_n$.

Moser proved in \cite{Mos} his Twist Map Theorem, which says that
an elliptic fixed point $p$ is stable, if there exists $n\ge 1$
such that the eigenvalue of $D_pf$ satisfies
$\lambda_p^{i}\neq 1$ for each $1\le i\le q$,
and $a_j(f^n,p)\neq 0$ for some $1\le j \le [n/2]-1$,
where $a_k, k\ge 1$ are the coefficients of Birkhoff normal form around $p$.
In this case, $p$ is also said to be Moser stable.
By perturbing the Birkhoff normal form
and then applying Moser twist map theorem,
Robinson proved in \cite{Rob70} that generically,
each elliptic periodic point is Moser stable.

It is expected that a small perturbation of the billiard table
will change the coefficients of Birkhoff normal form around an elliptic periodic point,
and turn that point into nonlinearly stable one.
However, it is quite difficult (if not impossible) to compute the Birkhoff normal form
for convex billiard dynamics on a convex sphere
with non-constant curvature,
since we do not know too much about the
explicit form around an elliptic periodic point,
and the dependence of $a_k(f^n,p)$ is quite involved
 (see \cite{DOP1,BuGr} for the planar case).

In the following we will take a different (simpler) approach
to improve the stability of an elliptic periodic points.
For an elliptic periodic point $p$, the rotation number $\rho$ of $p$ is
given by the rotation number of projective action
$[D_pF^n]$ on the projective space $\bP^1$.
Then $p$ is said to have Diophantine rotation number, if $\rho$
is Diophantine. That is, there exists positive numbers $c,\tau$
such that
\begin{equation}
\left|\rho-\frac{m}{n}\right|\ge \frac{c}{|n|^{2+\tau}},
\text{ for all rational numbers }\frac{m}{n}.
\end{equation}

The following is the so called Herman's {\it Last
Geometric Theorem}, which states that  an elliptic fixed point with Diophantine
rotation number is nonlinearly stable \cite{Yoc}.
See \cite{FK} for the history and a complete proof of Herman's LGT.
\begin{pro}\label{herman}
Let $f\in\mathrm{Diff}^\infty_\mu(M)$ and $p$ be an elliptic fixed point of $f$
with rotation number $\rho$.
If $\rho$ is Diophantine, then $p$ is stable.
\end{pro}
See \cite{HX13} for some applications of Herman's LGT
 to the study of the stability of Lagrangian equilibrium solutions of
circular restricted three body problems.

\begin{pro}\label{dense}
There is a $C^r$-dense subset $\cD_n\subset \cW_n\cap \Upsilon^\infty(S^2,g)$,
such that for each $\Gamma\in \cD_n$,
all elliptic periodic points in $P_n(\Gamma)$ are stable.
\end{pro}
\begin{proof}
Given a convex domain $\Gamma\in\cW_n\cap \Upsilon^\infty(S^2,g)$,
pick a sufficiently small
neighborhood $\cU\subset \cW_n$ of $\Gamma$ such that $P_n:\hat
\Gamma\in\cU\mapsto P_n(\hat\Gamma)$ has the same (finite) cardinality and
varies continuously. Note that each periodic point $p\in P_n(\Gamma)$ has
zero defect. We make a $C^r$-small and
$C^\infty$-smooth perturbation of $\Gamma$ around one point $p$
from each elliptic periodic orbit $\cO(p)$ in $P_n(\Gamma)$, say the
resulting domain $\hat\Gamma(\epsilon)$, such that the rotation number
$\rho_\epsilon$ of $p$ respecting the billiard map on $\hat\Gamma(\epsilon)$
is different from the initial rotation number, see Proposition
\ref{dege}. Note that  the set of Diophantine numbers has full measure on
the interval $(\rho,\rho_\epsilon)$ Picking a smaller size if necessary, we
can assume $\rho_\epsilon$ is already Diophantine.

Any two periodic orbits in $P_n(\Gamma)$ have no common reflection points.
So the above perturbation can be localized at one reflection point
and they have disjoint supports on $\Gamma$.
In particular the Diophantine rotation numbers
of the already perturbed ones
are preserved by the subsequent perturbations.

After a finite steps (at most $|P_n(\Gamma)|$) $C^r$-small and
$C^\infty$-smooth of perturbations,
we arrive at some $\hat\Gamma\in \cU\cap \Upsilon^\infty(S^2,g)$
such that $P_n(\hat\Gamma)=P_n(\Gamma)$,
$\hat F=F$ on $P_n(\hat\Gamma)$ and
 $\rho(p,\hat F)$ is Diophantine for each $p\in P_n(\hat\Gamma)$.
Then Proposition \ref{herman} guarantees that each elliptic periodic point
in $P_n(\hat\Gamma)$ is stable.
Such a perturbation $\hat\Gamma$ can be made arbitrarily $C^r$-close to $\Gamma$.
Therefore, $\cD_n$ is $C^r$-dense in $\cW_n$.
\end{proof}

\subsection{Homoclinic intersections}
Now we study the hyperbolic periodic points in $P_n(\Gamma)$.
Although each point $x\in P_n(\Gamma)$  is fixed by $F^n$,
the two branches of the stable (and unstable) manifolds $x$ may
be switched by $F^n$.
However, $F^{2n}$ does fix each branch of the invariant manifolds of
hyperbolic periodic points in $P_n(\Gamma)$.
When studying $P_n(\Gamma)$,  we actually consider
the $2n$-th iteration $F^{2n}$ of  those $\Gamma\in\cD_{2n}$.
For simplicity we denote $f=F^{2n}$.

Let $L$ is a branch of the unstable manifold $W^u(p)\backslash\{p\}$.
Then for any $x\in L$, the segment $L[x,fx]$ can be viewed as
a fundamental domain of $L$ with respect to $f=F^{2n}$.
As $k\to +\infty$, $f^{-k}L[x,fx]$ converges to $p$,
while $f^{k}L[x,fx]$ may have various limiting behaviors.
Denote by $\omega(L)$ the limit set of $f^{k}L[x,fx]$
as $k\to+\infty$. Similarly we define the $\omega$-set\footnote{Technically,
one should say
the $\alpha$-set of a stable branch.
We use the same notation for stable and unstable branches
just to unify the presentation of this paper.} of stable branches
(with respect to $f^{-k}$).
There is a dichotomy for the branches of invariant manifolds
(see \cite{Oli1}):
\begin{itemize}
\item either $\omega(L)\supset L$, or  $\omega(L)\cap L=\emptyset$.
\end{itemize}
A stronger dichotomy was obtained in \cite{XZ}.
\begin{pro}
Let $f\in \mathrm{Diff}_\mu(M)$ such that
each fixed point is nondegenerate,
and each elliptic fixed point is stable.
Let $L$ be a branch of invariant manifolds of a hyperbolic fixed point $p$.
Assume $fL=L$.
Then
\begin{itemize}
\item either $\omega(L)\supset L$,

\item or  $\omega(L)=\{q\}$ is a singleton, where $q$ is a hyperbolic fixed point.
\end{itemize}
\end{pro}
The branch $L$ with $\omega(L)=\{q\}$ is called
a saddle connection.
A  saddle connection is said to be a homoclinic (heteroclinic, respectively) connection
if $q=p$ ($q\neq p$, respectively).
\begin{proof}
We sketch the main idea of the proof. See \cite{XZ} for details.
Let $L$ be a branch of the unstable manifold of $p$. Suppose
$\omega(L)\not\supset L$. Then $\omega(L)\cap L=\emptyset$.
Let $K$ be the closure of $L$,
and $U$ be a connected component of $M\backslash K$
attached to $L$. Let $\hat U$ be the prime-end compactification of $U$,
whose boundary consists of  finitely many circles.
One of the circles, say  $C_p$, contains
the prime point $\hat p$ of $p$.
The restriction of $\hat f$ on $C_p$ is a circle diffeomorphism,
and admits $\hat p$ as an expanding fixed point.
So there is at least one more point on $C_p$ fixed by $\hat f$, say $\hat q$.
Let $q$ be the underlining point of $\hat q$ on the closure $\overline{U}$,
which must be fixed by $f$.
Such a point cannot be elliptic, since elliptic ones are stable
and cannot be approached by invariant curves outside $D_n$.
Then $q$ must be a hyperbolic fixed point,
and $L$ forms a branch of the stable manifold of $q$.
Therefore, $\omega(L)=\{q\}$.
\end{proof}

As a corollary, we obtain the following result
due to Mather \cite{Mat81}. Our formulation
is slightly stronger.
See also \cite[Corollary 3.4]{XZ}.
\begin{pro}\label{mather}
Let $f\in \mathrm{Diff}_\mu(M)$ such that
each fixed point of $f$ is either hyperbolic,
or elliptic with Diophantine rotation number.
Let $p$ be hyperbolic fixed point such that all four branches
of $W^{s,u}_{\pm}(p)$ are fixed by $f$.
Then either one of the branch forms a saddle connection,
or all four branches have the same closure.
\end{pro}
\begin{proof}
Pick a local coordinate system $(U,(x,y))$ around $p$ such that
the branches leave $p$ along the two axes.
Suppose none of the four branches is a  saddle connection.
Then each branch is recurrent, and its $\omega$-set contains
the branch itself and least one of the branches adjacent to it.
If the $\omega$-set of a branch $L$ does not contain the other adjacent
branch, say $K$, then we have $K\cap \overline{L}=\emptyset$.
Consider the component $V$ of $M\backslash \overline{L}$
containing $K$. Then there exists an smaller open neighborhood
$W\subset U$ of $p$, such that $\pa V\cap W$ consists
of the two pieces of the invariant manifolds of $p$.
One of the two pieces is from $L$, the other piece must be from $K_-$
(the other branch of the invariant manifold containing $K$).
on the other hand, we have $\pa V\cap W\subset L$.
Therefore $K_-=L$ forms a homoclinic loop,
which contradicts the hypothesis we started with.
\end{proof}

\begin{pro}\label{homo2n}
Let $\Gamma\in \cD_{2n}$. Then for each hyperbolic periodic point $x\in P_n(\Gamma)$,
there exist transverse homoclinic intersections between
each branch of the stable manifold and each branch of the unstable manifold
of $x$.
\end{pro}
The proof mainly use the fact that the (algebraic)
intersection number between two simple closed curves on $M$
must be 0.
This kind of arguments also appeared in \cite{Rob73, Pix,Oli1,XZ}.
\begin{proof}
Let $\Gamma\in \cD_{2n}$, $F$ be the induced billiard map
on $M=\Gamma\times (0,\pi)$.
Note that there is no saddle connection between
any hyperbolic periodic points  in $P_{2n}(\Gamma)$
(by the definition of $\cW_{2n}$, since
$\cD_{2n}\subset \cW_{2n}$),
and each elliptic periodic point in $P_{2n}(\Gamma)$ is stable
(by Proposition \ref{dense}, since $\cD_{2n}\subset \Upsilon^\infty(S^2,g)$).

Let $p$ be a hyperbolic periodic  point in $P_n(\Gamma)$,
$L$ be a branch of the unstable manifold of $p$,
and $K$ be a branch of the stable manifold of $p$.
Then both $L$, $K$ are fixed by $F^{2n}$, are recurrent,
and they have the same closure (by Proposition \ref{mather}).
Pick a local coordinate system $(U,(x,y))$ around $p$ such that
$L$ leaves $p$ along the positive $x$-axis,
and $L$ approximates $p$ through the first quadrant.
Let $S_{\epsilon}=\{(x,y)\in U: 0< x,y\le 1, xy\le \epsilon\}$,
and $q$ be the first moment on $L$ that hits $S_{\epsilon}$.
Let $C$ be the closed curve that starts from $p$, first travels along $L$
to the point $q$, and then the segment $\overline{qp}$ from $q$ to $p$.
Then $C$ is a simple closed curve.

\begin{figure}[h]
\begin{overpic}[width=60mm]{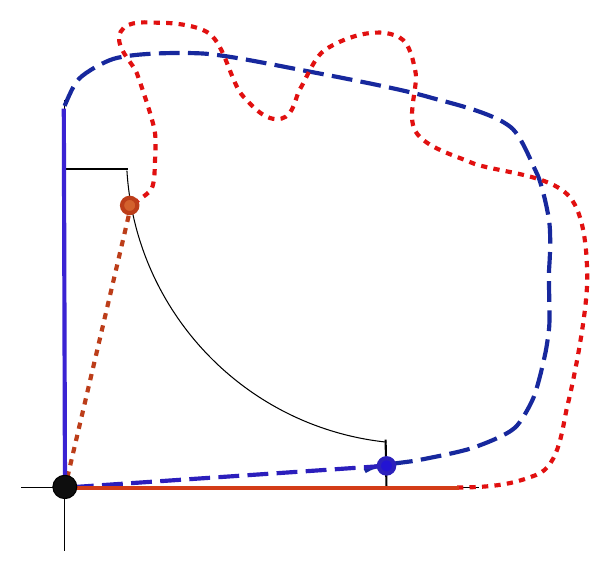}
\put(13,8){$p$}
\put(33,7){$L$}
\put(3,35){$K$}
\put(83,35){$\hat C$}
\put(98,35){$C$}
\put(20,25){$S_{\epsilon}$}
\put(25,60){$q$}
\put(65,21){$\hat q$}
\end{overpic}
\caption{The closing curves $C$ (red) and $\hat C$ (blue) when $K$ leaves along the positive
$y$-axis. The case that $K$ leaves along the negative $y$-axis is similar.}
\label{adjacent}
\end{figure}
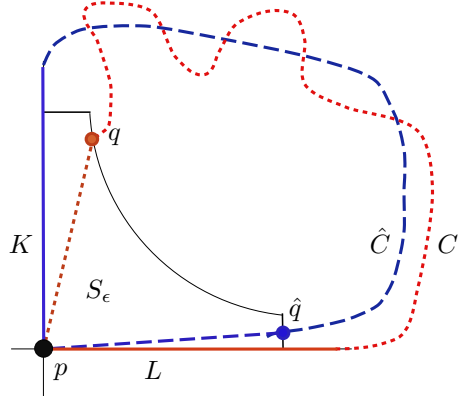

Since the closure of $K$
contains $L$, $K$ also intersects $S_{\epsilon}$.
Let $\hat C$ be the corresponding simple closed curve by closing
the first intersection $\hat q$ of $K$ with $S_{\epsilon}$.
Then we see that $C$ and $\hat C$ cross each other at $p$,
and the two open segments $(p,q)$ and $(p,\hat q)$ do not
intersect (by the entrance--exit analysis, see \cite{Oli1,XZ}).
Clearly $L(p,q)\cap (p,\hat q)=\emptyset$
and $K(p,\hat q)\cap (p,q)=\emptyset$.

However, the algebraic intersection number between any two closed curves
on $M$ must be 0. So $C$ and $\hat C$ have to cross  each other at
some point beside $p$, say $y$, and that intersection must happen between
$L(p,q)$ and $K(p,\hat q)$.
Therefore, there is a homoclinic intersection between $K$ and $L$.
The intersection at $y$ is a topological crossing, but may not be transverse.
However, transverse homoclinic intersections do exist,
since $\cD_{2n}\subset \cW_{2n}$.
\end{proof}
Note that no perturbation is needed for the proof of
the above proposition.

\begin{proof}[Proof of Proposition  \ref{homo1n}]
As we discussed right after stating Proposition
\ref{homo1n}, $\cX_n$ is open in $\cW_n$.
Let $\cD_{2n}$ be the dense subset of $\cW_{2n}$ given by Lemma \ref{dense}.
Then Proposition \ref{homo2n} shows that $\cD_{2n}\subset \cX_n$.
Therefore, $\cX_n$ is open
and dense in $\Upsilon^r(S^2,g)$. This completes the proof.
\end{proof}

\begin{proof}[Proof of Theorem \ref{main}]
Let $\cR^r=\bigcap_{n\ge 1}\cX_n$. Then $\cR^r$ is a residual subset
of $\Upsilon^r(S^2,g)$. For each $f\in \cR$,
and each hyperbolic periodic point $p$,
its stable and unstable manifolds admit some transverse intersections.
This completes the proof.
\end{proof}
The case $r=\infty$ can be proved
in the same way as we did for Theorem \ref{denseinfty}.

\subsection{Positive topological entropy}

\begin{coro}
There is an open and dense subset $\cU\subset \Upsilon^r(S^2,g)$ such that
for each $\Gamma\in\cU$, the billiard map has transverse homoclinic
intersections and positive topological entropy.
\end{coro}
\begin{proof}
Let $\cD_2$ be the dense subset given in Lemma \ref{dense}. Let
$\Gamma\in\cD_2$. Then each point $x\in P_2(\Gamma)$ is non-degenerate. Let
$W(s_1,s_2)= S(s_0,s_1)+S(s_1,s_2)$ be the action along the $2$-periodic
configuration $(s_k)$ on $\Gamma$. Let $(s_k)$ be an $2$-periodic
configuration at where $W$ attains its minimum, and $x$ the corresponding
periodic point of period $2$. Then $D^2W(s_1,s_2)$ is positive definite,
and $\mathrm{Tr}(D_xF^2)>2$ (see Proposition \ref{hessian}).
So $x$ is hyperbolic.
Moreover,  each branch of the
invariant manifolds of $x$ is fixed by $F^2$, since both eigenvalues are positive
(the double period iterate $F^{2n}$ is {\it not} needed for minimizers).
Then the proof of Proposition \ref{homo2n} shows that there exist transverse
homoclinic intersections of the stable and unstable branches of $x$.
Transverse intersections are robust. So there exists an open set $\cU\supset
\cD_2$ such that each $\Gamma\in \cU$ has  transverse homoclinic
intersections and  positive topological entropy.
\end{proof}

\appendix

\section{Zero defect for generic convex billiards}

In this section, we give a proof of Proposition \ref{zero}.
Let $\Upsilon^r(S^2,g)\subset C^r(\bT,S^2)$ be the set of convex curves.
We will use $f:\bT\to S^2$
to emphasize the role of $f$ as an embedding function,
and use $\Gamma=f(\bT)$ only for its image.
Let $f:\bT\to S^2$ be
a simple closed curve
enclosing a strictly convex domain $Q$,
$p$ be a nonsymmetric periodic point of the billiard map $F$
with period $n=|\cO(p)|\ge 3$.
Let $F^kp=(s_k,\theta_k)$, and
$\{y_1,\dots, y_t\}\subset \Gamma$ be the set of reflections of
$\cO(p)$ on $\Gamma$.
Suppose $p$ has positive defect: $d(p)=n-t>0$,
and $(y_{w(1)},\dots,y_{w(n)})$
be the ordered reflection sequence of $\cO(p)$.
This gives rise to an onto map
$w:\{1,\dots, n\}\mapsto \{1,\dots, t\}$.
Such a map $w$ is said to be the pattern of the orbit $\cO(p)$.
Without loss of generality we assume
$w(1)=1$ and $\{1\le j\le n: w(j)=1\}=\{j_1,\dots,j_r\}$
(with $j_1=1$) for some $r\ge 2$.

Now let $\cO(p)$ be a symmetric periodic orbit of period $n$.
Then $n=2m$ is an even number, and there are exactly
two reflections of right angle with $\Gamma$.
Suppose $p$ has positive defect, and
$t$ be the number of distinct reflection points on $\Gamma$.
Let $w:\{1,\dots, m,m+1\}\to \{1,\dots, t\}$
be the pattern of $\cO(p)$
such that $y_{w(1)}$ and $y_{w(m+1)}$ are the two
reflection points on $\Gamma$ with right angle.
Note that $w(1)\neq w(m+1)$.
We first study the nonsymmetric case in details.
The symmetric case
need some minor modifications, and will be given at the end of the proof.

Now we generalize above notations to
any closed path of type $w$
on $S^2$. Let $t\ge 3$ be given.
Then a map $w:\bZ\to \{1,\dots, t\}$
is said to be of period $n$ if $w(k+n)=w(k)$ for all $k$;
is said to be admissible if $w(k)\neq w(k+1)$ for all $k$.
There are only finitely many admissible patterns of period $n$,
and we will fix such a pattern from now.
Let $\bT^{(t)}\subset \bT^t$
be the set of points $(s_1,\dots, s_t)$ with $s_i\neq s_j$
for all $i\neq j$.
Then for each $\bx\in \bT^{(t)}$ and $\by=f^{(t)}(\bx)$,
we have that $\{y_{w(k)}\}$ is a closed path  of type $w$.

Let $\by=(y_1,\dots, y_t)$ be a collection of $t$ distinct points
on $S^2$. Define the perimeter of the geodesic polygon
with the ordered corners at $\{y_{w(k)}\}$ as
$$H_w(\by)=\sum_{k=1}^n d(y_{w(k)},y_{w(k+1)}).$$
Similarly, given $f:\bT\to S^2$ and $\bx\in\bT^{(t)}$,
let $\ds H_w(f^{(t)}\bx)$ be the perimeter of the corresponding
geodesic polygon with corners $(f(s_i))$ and pattern $w$.

Let $J^1(\bT,S^2)$ be the 1-jet bundle, and $J^1_t(\bT,S^2)$
be the $t$-fold jet bundle. For each $f\in C^r(\bT,S^2)$, we have
a section map $j_tf:\bx\in\bT^{(t)}\mapsto(jf(s_1),\dots, jf(s_t))$.
Let $V_w$ be the set of those $\tau=(jf_1(s_1),\dots, jf_t(s_t))$
such that
\begin{enumerate}
\item $f_{j}(s_j)\neq f_i(s_i)$ for each $j\neq i$,

\item $f_i'(s_i)\neq 0$ for every $i=1,\dots,t$, and

\item  the polygon generated by
$(f_1(s_1),\dots, f_t(s_t))$  is convex with $t$ vertices.
\end{enumerate}

Let $\alpha:J^1(\bT,S^2)\to \bT$ be the source map,
and $\beta:J^1(\bT,S^2)\to S^2$ be the target map,
$W:=(\alpha^t)^{-1}(\bT^{(t)}) \;\cap\; V_w$.
Clearly $W$ is an open submanifold of $J^1_t(\bT,S^2)$.
Given $\tau\in W$, there are neighborhoods $U_i\subset \bT$ of
$s_i$ and $V_i\subset S^2$ of $f_i(U_i)$
with $U_i\cap U_j=\emptyset$ and $V_i\cap V_j=\emptyset$
whenever $1\le i< j\le t$,
 such that
\[\Omega:=W\cap \left(\prod_{i=1}^t J^1(U_i,V_i)\right)\]
is an open neighborhood of $\tau$.
Consider the coordinate map
\[\theta:\Omega\mapsto \prod_{i=1}^t U_i\times T_{V_i}S^2
\simeq \prod_{i=1}^s U_i\times V_i\times \bR^2,\]
with $\theta(\tau)=(\bu,\bv,A)$,
where $\bu=(u_1,\dots,u_t)$ is the source of $\tau$,
$\bv=(v_1,\dots,v_t)$ is the target of $\tau$,
and $\ds A=\left(f_{1}'(u_1), \dots, f_{i}'(u_i), \dots, f_{t}'(u_t)\right)
=\begin{pmatrix}
f_{1,1}'(u_1) & \dots & f_{i,1}'(u_i) & \dots & f_{t,1}'(u_t)\\
f_{1,2}'(u_1) & \dots & f_{i,1}'(u_i) & \dots & f_{t,2}'(u_t)
\end{pmatrix}$.

In the following we separate the role of $s_1$ from $s_k$, $2\le k\le t$.
For each $l=1,\dots, r$, let $a=w(j_l-1)$ and $b=w(j_l+1)$,
and $\eta_a$ be the tangent direction of the shortest geodesic
from $y_1$ to $y_a$, and similarly define $\eta_b$.
Let $\bt_{y_1}=f_1'(u_1)$ and $\bn_{y_1}$
be the unit tangent and normal directions
at $y_1$.
Then we decompose $\eta_a+\eta_b$ as
\[\eta_a+\eta_b=\xi_l(\by)\bt_{y_1}+\zeta_l(\by)\bn_{y_1}, \]
where $\xi_l(\by)=\langle\eta_a+\eta_b ,\bn_{y_1}\rangle$
and  $\zeta_l(\by)=\langle\eta_a+\eta_b ,\bt_{y_1}\rangle$.
Then it follows from the basic properties of billiard maps that
\begin{enumerate}
\item[(1a).] $\zeta_l(\by)=0$ if $\ds (y_{w(t)})_{t=1}^n$ is a periodic orbit;

\item[(1b).] $\zeta_l(\by)\neq 0$ if $(y_a,y_1,y_b)$ does not describe
a reflection.

\item[(2a).]  $\pa_{s_k}H\circ f^{(t)}(\bx)=0$ for orbit paths;

\item[(2b).] $\pa_{s_k}H\circ f^{(t)}(\bx)$ may not be zero for non-orbit paths.
\end{enumerate}

Let $\Sigma_{w}\subset M$ be those
$\tau=(jf_1(s_1),\dots, jf_t(s_t))\in M$
so that $\zeta_l(\by)=0$ for each $1\le l\le r$, and
$\pa_{s_k}F_w\circ (f_1,\dots,f_t)(\bx)=0$ for each $2\le k\le t$.
We first estimate the codimension of $\Sigma_w$.
Let $\tau\in \Sigma_w\subset M$ be given,
and $\theta:\tau\mapsto (\bu,\bv,A)$
be the coordinate system around $\tau$ given as above.
Define a function
\[\cK: \theta(\Omega)\to \bR^{r+t-1}, \quad
\chi\mapsto (\phi_1(\chi),\dots,\phi_r(\chi);
\psi_2(\chi),\dots, \psi_t(\chi)),\]
where
\begin{enumerate}
\item $\phi_l : \theta(\Omega)\to\bR$, $l=1,\dots,r$,
is defined by
\[\chi=(\bu,\bv,A)\mapsto \zeta_l(\by)
=\langle\eta_a+\eta_b ,\bt_{y_1}\rangle,\]
where $a=w(j_l-1)$, $b=w(j_l+1)$,
and $\bt_{y_1}$ is the unit tangent direction
along $f_1(u_1)$.

\item $\ds \psi_{k}:\theta(\Omega)\to\bR$,
$\chi\mapsto \langle\nabla_{y_k}H, \bt_{y_k}\rangle$,
for each $k=2,\dots,t$.
\end{enumerate}

Note that $\cK(\tau)=0$ for each $\tau\in \Sigma_w\cap \Omega$.
We claim that $\cK$ is a submersion at each point in $\Omega$.
The verification of the submersion is pretty simple for convex billiards:
by pushing the point $y_a$ along the normal direction of $f_a(s_a)$
(for $a=w(j_l-1)$, while fixing all other $y_k$, $k\neq a$),
we see that $\phi_l$ changes linearly (since $\bt_{y_1}$ is fixed);
by rotating the tangent direction $\bt_{y_k}$ of  $y_k$ along $f_k(s_k)$,
(while fixing all $y_k$),
we see that $\psi_k$ changes linearly
(since $\nabla_{y_k}F$ is a fixed nonzero vector);
and all these variations are independent.

Therefore, the map $\cK$ is a submersion at each point in $\Omega$.
So the codimension of $\Sigma_w$ in $\Omega\subset W$
is at least $\dim(\text{Im}(\cK))=r+t-1\ge t+1$, which is larger than
$\dim \bT^{(t)}=t$. Then by Multi-jet Transversality Theorem, we have that
$j_tf \cap \Sigma_{w}=\emptyset$ for a residual subset of convex tables.
Similarly, we define $\Sigma_{w'}$ for any
$n$-periodic admissible pattern $w':\bZ\to \{1,\dots,t\}$,
and then for any $t=2,\dots, n-1$.
This completes the proof for nonsymmetric periodic orbits.\\

For symmetric periodic orbits, the proof is almost the same.
The only difference is that when $a:=w(j_l-1)=w(j_l+1)$,
and the collision from $y_a$ to $y_1$ is at the right angle. In this case,
we still have that $\phi_l$ changes linearly
by pushing $y_a$  along the normal direction of $f_a(s_a)$
(since $\bt_{y_1}$ is fixed).
Then the rest of the proof is the same.
Putting together these results, we get that,
for a residual subset of convex tables,
each periodic orbits with period $n$ has zero defect.
This completes the proof that genericity of zero defect.\\

For the second part of Proposition \ref{zero}, we note that
in the proof given above, we used the property that
each folding of the path at $y_{w(k)}$ is a reflection;
but we did not use any property that $\{y_{w(k)}\}$
is on a single orbit. In particular, one can take the union
of the two periodic orbits and then study the paths with that joint pattern.
Therefore the same analysis applies to
the case that two orbits have some common reflection point.
Then we conclude that, there is a residual subset of convex tables,
for which any two periodic orbits with no common geodesic segment has no common
reflection point. However, note that the orbit obtained by the
time-reversal of one orbit has exact the same  geodesic segments,
and this does not count as positive defects.

\section*{Acknowledgments}
The author thanks Leonid Bunimovich, Victor Donnay,
Alex Grigo, Sa\v{s}a Koci\'c, Jeff Xia and Hong-Kun Zhang for
many valuable comments and useful discussions.
The author is very grateful to the anonymous referees for many useful
comments and suggestions,
which helped him to improve the presentation of the paper significantly.

\end{document}